\documentclass{siamart_arXiv}

\usepackage{lipsum}
\usepackage{amsfonts}
\usepackage{graphicx}
\usepackage{epstopdf}
\usepackage{algorithmic}
\ifpdf
  \DeclareGraphicsExtensions{.eps,.pdf,.png,.jpg}
\else
  \DeclareGraphicsExtensions{.eps}
\fi

\newcommand{\TheTitle}{Regularization based on all-at-once formulations for inverse problems} 
\newcommand{\ShortTitle}{Regularization based on all-at-once formulations} 
\newcommand{\TheAuthors}{B. Kaltenbacher}

\headers{\ShortTitle}{\TheAuthors}

\title{{\TheTitle}\thanks{This work was supported by the Austrian Science Fund FWF under grant I2271.}}

\author{
  Barbara Kaltenbacher\thanks{Alpen-Adria-Universit\"at Klagenfurt, Austria
    (\email{barbara.kaltenbacher@aau.at}, \url{http://wwwu.uni-klu.ac.at/bkaltenb/}).}
}

\usepackage{amsopn}

\theoremstyle{remark}
\newtheorem{example}{Example}
\newtheorem{remark}{Remark}
\newtheorem{assumption}{Assumption}

\renewcommand{\phi}{\varphi}

\renewcommand{\ast}{\star}
\newcommand{\R}{\mathbb{R}}
\newcommand{\N}{\mathbb{N}}
\renewcommand{\bar}{\overline}
\newcommand{\calB}{\mathcal{B}}
\newcommand{\calC}{\mathcal{C}}
\newcommand{\calD}{\mathcal{D}}
\newcommand{\calR}{\mathcal{R}}
\newcommand{\calS}{\mathcal{S}}
\newcommand{\bfF}{\mathbf{F}}
\newcommand{\bfR}{\mathbf{R}}
\newcommand{\bfx}{\mathbf{x}}
\newcommand{\bfy}{\mathbf{y}}
\newcommand{\xdag}{x^\dagger}
\newcommand{\udag}{u^\dagger}
\newcommand{\ydel}{y^\delta}
\newcommand{\xad}{x_\alpha^\delta}
\newcommand{\uad}{u_\alpha^\delta}
\newcommand{\wad}{w_\alpha^\delta}
\newcommand{\xadd}{x_{\alpha(\delta)}^\delta}
\newcommand{\uadd}{u_{\alpha(\delta)}^\delta}
\newcommand{\xadb}{\bar{x}_\alpha^\delta}
\newcommand{\uadb}{\bar{u}_\alpha^\delta}
\newcommand{\wadb}{\bar{w}_\alpha^\delta}
\newcommand{\xadbd}{\bar{x}_{\alpha(\delta)}^\delta}
\newcommand{\uadbd}{\bar{u}_{\alpha(\delta)}^\delta}
\newcommand{\xkd}{x_k^\delta}
\newcommand{\ukd}{u_k^\delta}
\newcommand{\ukdb}{\bar{u}_k^\delta}
\newcommand{\wkdb}{\bar{w}_k^\delta}
\newcommand{\zkdb}{\bar{z}_k^\delta}
\newcommand{\xkpd}{x_{k+1}^\delta}
\newcommand{\ukpd}{u_{k+1}^\delta}
\newcommand{\wkpd}{w_{k+1}^\delta}
\newcommand{\xkpdb}{\bar{x}_{k+1}^\delta}
\newcommand{\ukpdb}{\bar{u}_{k+1}^\delta}
\newcommand{\wkpdb}{\bar{w}_{k+1}^\delta}
\newcommand{\xskd}{x_k^{*\delta}}
\newcommand{\uskd}{u_k^{*\delta}}
\newcommand{\xkdb}{\bar{x}_k^{\delta}}
\newcommand{\xskdb}{\bar{x}_k^{*\delta}}

\newcommand{\xskpd}{x_{k+1}^{*\delta}}
\newcommand{\uskpd}{u_{k+1}^{*\delta}}
\newcommand{\xskpdb}{\bar{x}_{k+1}^{*\delta}}

\newcommand{\vobs}{v^{\text{\footnotesize{obs}}}}
\newcommand{\vmod}{v^{\text{\footnotesize{mod}}}}
\newcommand{\yobs}{y^{\text{\footnotesize{obs}}}}
\newcommand{\ymod}{y^{\text{\footnotesize{mod}}}}
\newcommand{\norm}[1]{\left\|#1\right\|}
\newcommand{\ul}[1]{\underline{#1}}
\newcommand{\ol}[1]{\overline{#1}}
\newcommand{\setof}[2]{\left\{#1:#2\right\}}

\usepackage{graphicx}

\ifpdf
\hypersetup{
  pdftitle={\TheTitle},
  pdfauthor={\TheAuthors}
}
\fi

\begin{document}

\maketitle

\begin{abstract}
Parameter identification problems typically consist of a model equation, e.g. a (system of) ordinary or partial differential equation(s), and the observation equation. In the conventional reduced setting, the model equation is eliminated via the parameter-to-state map. Alternatively, one might consider both sets of equations (model and observations) as one large system, to which some regularization method is applied. The choice of the formulation (reduced or all-at-once) can make a large difference computationally, depending on which regularization method is used: Whereas almost the same optimality system arises for the reduced and the all-at-once Tikhonov method, the situation is different for iterative methods, especially in the context of nonlinear models. In this paper we will exemplarily provide some convergence results for all-at-once versions of variational, Newton type and gradient based regularization methods. Moreover we will compare the implementation requirements for the respective all-at-one and reduced versions and provide some numerical comparison.
\end{abstract}



\section{Introduction}
In their original formulation, inverse problems often consist of a model and additional observations. Consider, e.g., an equation (PDE, ODE, integral equation) model for the state $u$ 
\begin{equation}\label{Axu}
A(x,u)=0
\end{equation}
containing a parameter $x$ (or a set of parameters) that is to be determined from additional observations of the state
\begin{equation}\label{Cuy}
C(u)=y
\end{equation}
Here $A:\calD(A)(\subseteq X\times V)\to W^*$ and $C:\calD(C)(\subseteq V)\to Y$ are operators acting between Banach spaces $X$, $V$, $W^*$, $Y$ (the star indicates that in variational formulations of models $W^*$ will typically be the dual of some Banach space).
The setting could be extended in several directions, e.g., the observation can as well depend on some unknown parameters that have to be identified or the model could consist of a variational inequality instead of an equation. 
Still \eqref{Axu}, \eqref{Cuy} is sufficiently general to comprise a wide range of applications, e.g., the following examples.

\subsection{Examples}
\begin{example} \label{ex1}
Consider a boundary value problem for a linear elliptic PDE on a smooth bounded domain $\Omega\subseteq\R^d$, $d\in\{1,2,3\}$  
\[
-\nabla (a\nabla u)+cu=b \mbox{ in }\Omega\,, \quad \frac{\partial u}{\partial n}=g\mbox{ on }\partial\Omega
\]
with a given boundary excitation $g\in H^{-1/2}(\partial\Omega)$,
and possibly spatially varying coefficients $a,b,c$ and the inverse problem of identifying these coefficients $a,b,c$ (or part of them) from additional measurements $C(u)$ of the PDE solution $u$ on (part of) the domain or on its boundary. This fits into the above framework with, e.g., $X=W^{1,p}(\Omega)\times L^\infty(\Omega)\times L^p(\Omega)$, $p\in[1,\infty]$, $V=H^1(\Omega)=W$, (or, if $\Omega$, $g$  are sufficiently smooth, 
$V=W^{2,p}(\Omega)$, $W^*=L^p(\Omega)$, $p\in[1,\infty]$,) 
\[
\langle A(a,b,c,u),w\rangle_{W^*,W}=\int_\Omega \Bigl(a\nabla u \cdot \nabla w +cuw-bw\Bigr)\, dx-\int_{\partial\Omega} g w\, ds
\] 
and $Y=L^q(\Sigma)$, $p\in[1,\infty]$, $C(u)=u\vert_\Sigma$, where $\Sigma \subseteq\overline{\Omega}$ is an open subdomain of $\Omega$ or a regular curve/surface contained in its boundary or in its interior, so that an embedding or a trace theorem yields continuity of the observation map $C:H^1(\Omega)\to L^p(\Sigma)$.
\end{example}

\begin{example} \label{ex2}
Using similar measurements but a nonlinear model, we consider identification of the nonlinearity, i.e., the function $q$ in the elliptic boundary value problem
\[
-\Delta u+q(u)=0 \mbox{ in }\Omega\,, \quad \frac{\partial u}{\partial n}=g\mbox{ on }\partial\Omega
\]
with given $g$.
Here we use a space $X$ that is continuously embedded in $C[\underline{u},\overline{u}]$ for an interval  $[\underline{u},\overline{u}]$ containing all possibly appearing values of $u$ (which can, e.g., be estimated by using maximum principles in case the PDE is elliptic, depending on the monotonicity of $q$), 
$V=H^1(\Omega)=W$, 
\[
\langle A(q,u),w\rangle_{W^*,W}=\int_\Omega \Bigl(\nabla u \cdot \nabla w +q(u)w\Bigr)\, dx-\int_{\partial\Omega}gw\, ds
\] 
and $Y$, $C$ as in Example \ref{ex1} above.
\end{example}

\begin{example} \label{ex2b}
Alternatively, one often encounters inverse source problems for nonlinear PDEs such as the simple model example of identifying $b$ in 
\[
-\Delta u+\zeta u^3=b \mbox{ in }\Omega\,, \quad \frac{\partial u}{\partial n}=g\mbox{ on }\partial\Omega
\]
where $g$ and $\zeta$ are given.
Here we have $X=L^p(\Omega)$, $p\in[1,\infty]$, $V=H^1(\Omega)=W$, 
\[
\langle A(b,u),w\rangle_{W^*,W}=\int_\Omega \Bigl(\nabla u \cdot \nabla w +\zeta u^3 w-bw\Bigr)\, dx-\int_{\partial\Omega}gw\, ds\,,
\] 
and again $Y$, $C$ as in Example \ref{ex1} above.
\end{example}

\begin{example} \label{ex3}
Consider identification of the (possibly infinite dimensional) parameter $\vartheta$ in the state space system consisting of an ODE model and observations
\[
\begin{aligned}
&\dot{u}(t)=f(t,u(t),\vartheta)\ t\in(0,T)\,, \quad u(0)=u_0\\
&y=C(u)\,,
\end{aligned}
\]
where the dot denotes time differentiation, $f:(0,T)\times\R^n\times X\to\R^n$ is a given function and $u_0\in\R^n$ is a given initial value. Using semigroup theory, this could as well be extended to time dependent PDEs. An example of an infinite dimensional stationary parameter $\vartheta$ to be identified in a system of ODEs is the Preisach weight function in some hysteretic evolutionary model.
The observations $C(u)$ are, e.g. discrete or continuous in time 
$y_i=g_i(u(t_i))$, $i\in\{1,\ldots,m\}$ (including the case of final measurements $t_i=T$) or $y(t)=g(t,u(t))$, $t\in(0,T)$ with given functions $g_i:\R^n\to\R^{mM}$ or $g:(0,T)\times\R^n\to L^2(0,T;\R^M)$.
\end{example}

\subsection{Motivation of the all-at-once-approach}
Definition and analysis of solution methods for such inverse problems is often based on a reduced formulation that is obtained by the use of a parameter-to-state map i.e., a mapping $S:\calD(S)(\subseteq X)\to V$, that resolves \eqref{Axu} with respect to $x$
\begin{align*}
\forall x\in \calD(S): \ & A(x,S(x))=0 \text{ and } \\
&\forall u\in V: ((x,u)\in\calD(A) \text{ and } A(x,u)=0) \ \Rightarrow \ u=S(x).
\end{align*}
Existence of such a mapping is guaranteed by the Implicit Function Theorem if for an open set $\calB$ with $\calD(S)\subset\calB\subseteq X$, $A$ is continuously Fr\'{e}chet differentiable on $\calB\times V$ and its derivative $A_u$ with respect to the state is boundedly invertible with uniform bound: 
\begin{assumption}\label{ass:A_u}
\[
\exists C_A \ \forall (x,u)\in (\calB\times V) \cap \calD(A)\,: \ \norm{A_u(x,u)^{-1}}\leq C_A\,.
\]
\end{assumption}
In order to satisfy this assumption, usually the domain of $A$ has to be restricted, e.g., to 
\begin{equation}\label{domA:ex1}
\calD(A)\subseteq\setof{(x,u)=(a,b,c,u)}{\underline{a}\le a\le \overline{a} \mbox{ a.e. on }\Omega, \ c\geq\ul{c} \mbox{ a.e. on }\Omega}
\end{equation}
with positive constants $0<\ul{a}<\ol{a}$, $0<\ul{c}$ 
in example \ref{ex1} or to 
\begin{equation}\label{domA:ex2}
\calD(A)\subseteq\setof{(x,u)=(q,u)}{\ul{q}\le \frac{q(\tilde{\lambda})-q(\lambda)}{\tilde{\lambda}-\lambda}\ \forall \tilde{\lambda}\not=\lambda\in\R}
\end{equation}
with some constant $\ul{q}>0$ 
in example \ref{ex2}.

Under such conditions the forward operator $F:\calD(F)(\subseteq X)\to Y$, $F=C\circ S$ is well-defined on $\calD(F)=\calD(S)$ and the inverse problem \eqref{Axu}, \eqref{Cuy} can be equivalently written as an operator equation
\begin{equation}\label{Fxy}
F(x)=y
\end{equation}
Such problems are typically ill-posedness in the sense that $F$ is not continuously invertible
and instead of $y$ usually only a noisy version $\ydel$ is available, which we here assume to obey the deterministic and known noise bound $\delta$
\begin{equation}\label{delta}
\norm{y-\ydel}\leq\delta\,,
\end{equation}
thus regularization (see, e.g., \cite{BakKok04,EHNbuch96,KalNeuSch08,Kirs96,SGGHL08,SKHK12,Voge02} and the references therein) has to be employed.

In this paper we will return to the original formulation as an all-at-once system of model and observation equation \eqref{Axu}, \eqref{Cuy}, 
\begin{equation}\label{Fxuy}
\bfF(\bfx)=\bfF(x,u)=\left(\begin{array}{c}A(x,u)\\C(u)\end{array}\right)
=\left(\begin{array}{c}0\\y\end{array}\right)=\bfy
\end{equation}
and investigate the behaviour of some well-known regularization paradigms when applied to $\bfF$ instead of $F$. We will see that this enables to avoid restrictions like \eqref{domA:ex1}, \eqref{domA:ex2} and, moreover, can make a considerable difference when it comes to implementation. 
We will also provide a convergence analysis that goes beyond the mere application of known results to the operator $\bfF$ in the sense that regularization might not just be applied to the whole of $\bfx=(x,u)$ but - as natural - only to the $x$ part, a case which requires some extra considerations in the convergence analysis.

All-at-once approaches to inverse problems have already been considered previously, see, e.g., 
\cite{BurgerMuehlhuberIP,BurgerMuehlhuberSINUM,HaAs01,KKV14b,LeHe16}.
While these papers concentrate on computational aspects and convergence analysis of particular methods, our aim is here to provide a comparative overview on several regularization paradigms.

In the remainder of this paper we will assume that a solution 
\begin{equation}\label{defDbfF}
(\xdag,\udag)\in\calD(\bfF)\subseteq\calD(A)\cap(X\times\calD(C))
\end{equation}
to \eqref{Fxuy} exists and that $\calD(\bfF)$ is convex. Note that $\calD(\bfF)$ need not necessarily be the maximal domain of $\bfF$, so restriction to a convex set can be done without loss of generality here.

While data misfit and regularization terms will be defined by norms for simplicity of exposition, most of the results are extendable to more general discrepancy and regularization functionals as considered, e.g. in \cite{DissFlemming,DissPoeschl,DissWerner}. 

We treat methods and convergence conditions only exemplarily to highlight analogies and differences between reduced and all-at-once formulations, so our aim is not to provide a complete convergence analysis (a priori and a posteriori choice of regularization parameters, general rates, etc.) for each of the discussed methods. 

The remainder of this paper is organized as follows. In Sections \ref{sec:Tikh}, \ref{sec:IRGNM}, \ref{sec:LW}, we consider Tikhonov, Newton type, and gradient type regularization, respectively. For each of these three paradigms we provide some convergence results with a priori and a posteriori regularization parameter choice strategies and under different assumptions on the forward operator (This analysis part is restricted to just quotation of a convergence result in the gradient method case.) Moreover we compare the key implementation requirements.  Section \ref{sec:numex} contains a verification of the convergence conditions for the iteratively regularized Gauss-Newton method from Section \ref{sec:IRGNM} for Example \ref{ex2b}, whose all-at-once and reduced versions are then compared numerically. Some preliminary considerations on comparison of convergence conditions for all-at-once and reduced versions are provided in Section \ref{sec:rem}, and we give a short summary and an outlook in Section \ref{sec:concl}.

\paragraph{Notation}
For some $r\in[1,\infty)$ we denote by $r^*=\frac{r}{r-1}$ the dual index.
$J_r^X=\partial\frac{1}{r}\norm{\cdot}_X^r$ denotes the duality mapping with gauge function $\frac{1}{r}t^r$, which is in general set valued. For smooth spaces, i.e., spaces with G\^{a}teaux differentiable norm on the unit sphere, $J_r^X(x)$ will be single valued; otherwise, by a slight abuse of notation, we denote by $J_r^X(x)$ a single valued selection from this set.

\section{Tikhonov regularization}\label{sec:Tikh}

For any $\rho,\alpha>0$, $m,o,r\in[1,\infty)$, $x_0\in\calB$, $u_0\in V$, we define the pair $(\xad,\uad)$ as a minimizer of 
\begin{equation}\label{minTikh}
\min_{(x,u)\in\calD(\bfF)} \calS(\bfF(x,u),(0,\ydel))+\alpha \calR(x,u)
\end{equation}
with 
\begin{equation}\label{defS}
\calS((w^*,y),(\ymod,\yobs))=\frac{\rho}{m}\norm{w^*-\ymod}_{W^*}^m+\frac{1}{o}\norm{y-\yobs}_Y^o
\end{equation}
and 
\begin{equation}\label{Tikh_aao}
\calR(x,u)=\calR_1(x)=\frac{1}{r}\norm{x-x_0}_X^r
\end{equation}
or 
\begin{equation}\label{Tikh_aao_u}
\calR(x,u)=\calR_2(x,u)=\frac{1}{r}\norm{x-x_0}_V^r+\frac{1}{q}\norm{u-u_0}_V^q\,,
\end{equation}
Note that in case of \eqref{Tikh_aao} regularization is only applied to $x$.
In both versions, $\rho$ will remain fixed, whereas $\alpha$ will be chosen in dependence of $\delta$ - typically in such a way that it tends to zero as $\delta\to0$.

\subsection{Well-definedness and convergence}
The analysis will be based on the following assumptions 
\begin{assumption} \label{ass:compactXV}
Bounded sets in $X$, $V$ are weakly (or weakly*) compact.
\end{assumption}
\begin{assumption} \label{ass:bfFweaklyclosed}
$\bfF$ is weakly (or weakly*) sequentially closed, i.e.,
\[
\begin{aligned}
\forall ((x_k,u_k))_{k\in\N}\subseteq\calD(\bfF): \quad& \Bigl(x_k\rightharpoonup x\,, \ u_k\rightharpoonup u\,, \ A(x_k,u_k)\rightharpoonup f\,, C(u_k)\rightharpoonup y \Bigr) \\ 
&\Longrightarrow \  \Bigl( (x,u)\in\calD(\bfF)\text{ and } A(x,u)=f\,, \ C(u)=y\Bigr)\,,
\end{aligned}
\] 
\end{assumption}
where $\rightharpoonup$ denotes weak or weak* convergence.

For only proving convergence provided minimizers are already well-defined the following somewhat weaker assumption (note that we have strong convergence of the images in the premiss) suffices in place of Assumption \ref{ass:bfFweaklyclosed}.
\begin{assumption} \label{ass:bfFweakstrongclosed}
\[
\begin{aligned}
\forall ((x_k,u_k))_{k\in\N}\subseteq\calD(\bfF): \quad& \Bigl(x_k\rightharpoonup x\,, \ u_k\rightharpoonup u\,, \ A(x_k,u_k)\to f\,, C(u_k)\to y \Bigr) \\ 
&\Longrightarrow \  \Bigl( (x,u)\in\calD(\bfF)\text{ and } A(x,u)=f\,, \ C(u)=y\Bigr)
\end{aligned}
\] 
\end{assumption}

Assumption \ref{ass:compactXV} is satisfied if $X,V$ are reflexive or duals of separable normed spaces. (The star in weak* will be skipped in the following.)
Sufficient for Assumption \ref{ass:bfFweaklyclosed} is weak continuity of $A$ and $C$ and weak closedness of $\calD(\bfF)$.

To prove well-definedness of a minimizer and convergence in case of regularization with respect to $x$ only, \eqref{Tikh_aao}, we additionally impose Fr\'{e}chet differentiability of $A$ and $C$, Assumption \ref{ass:A_u} and a growth condition on the derivative $A_x$ of the model operator with respect to the parameter: 
\begin{assumption}\label{ass:A_x}
There exists a (without loss of generality monotonically increasing) function $\psi:\R^+\to\R^+$ such that
\[
\begin{aligned}
&(a) \quad \forall (x,v)\in (\calB\times V) \cap \calD(A)\,: \ 
\norm{A_x(x,u)}\leq \psi\left(\norm{x}_X\right)\left(1+\norm{u}_V\right) \\
&\qquad\text{ and } \psi\left(\norm{\xdag}_X+\lambda\right)\lambda\to0 \text{ as }\lambda\to0\text{ and }\norm{\xdag-x_0}_X\text{ is sufficiently small}\\
&\text{or}\\
&(b) \quad \forall (x,v)\in (\calB\times V) \cap \calD(A)\,: \ 
\norm{A_x(x,u)}\leq \psi\left(\norm{x}_X\right)
\end{aligned}
\]
\end{assumption}
\begin{proposition}\label{th:convTikh}
Let Assumptions 
\ref{ass:compactXV} \ref{ass:bfFweaklyclosed}, and in case of \eqref{Tikh_aao} additionally Assumptions \ref{ass:A_u}, \ref{ass:A_x} be satisfied. 
\begin{itemize}
\item Then there exists $\bar{\delta}>0$ such that for all $\delta\in(0,\bar{\delta})$, $\alpha>0$ a minimizer of \eqref{minTikh} with \eqref{defS} and \eqref{Tikh_aao} or \eqref{Tikh_aao_u} exists.
\item
These minimizers are stable with respect to the data $\ydel$ in the sense that for any sequence $(y_k)_{k\in\N}$ converging to $\ydel$ in the $Y$ norm and any sequence of corresponding minimizers $((x_k,u_k))_{k\in\N}$ there exists a weakly convergent subsequence and the limit of every weakly convergent subsequence is a minimizer of \eqref{minTikh}.
\item If $\alpha=\alpha(\delta)$ is chosen a priori according to 
\begin{equation}\label{apriori}
\alpha\to0\text{ and }\frac{\delta^o}{\alpha}\to0\text{ as }\delta\to0
\end{equation}
or a posteriori according to the generalized discrepancy principle 
\begin{equation}\label{aposteriori}
\frac{\delta^o}{o}\leq 
\calS(\bfF(\xad,\uad),(0,\ydel))
\leq \frac{\tau\delta^o}{o}
\end{equation}
for some fixed $\tau>0$ independent of $\delta\in(0,\bar{\delta})$,
then there exists $\bar{C}>0$ independent of $\delta\in(0,\bar{\delta})$ such that for any solution $(\xdag,\udag)$ of \eqref{Axu}, \eqref{Cuy} we have boundedness
\begin{equation}\label{xaddxdag}
\begin{aligned}\forall\delta\in(0,\bar{\delta})\, : & \ 
\frac{1}{r}\norm{\xadd-x_0}_X^r
\leq \begin{cases}
\frac{1}{r}\norm{\xdag-x_0}_X^r +\frac{\delta^o}{o\alpha(\delta)}
\text{ with \eqref{apriori}}\\
\frac{1}{r}\norm{\xdag-x_0}_X^r \text{ with \eqref{aposteriori}}\end{cases}
\\& \text{ and }\norm{\uadd}_V\leq \bar{C}
\end{aligned}
\end{equation}
in case of minimizers of \eqref{minTikh} with \eqref{defS} and \eqref{Tikh_aao} and 
\begin{equation}\label{xaddxdag_u}
\forall\delta\in(0,\bar{\delta})\, : \ 
\calR_2(\xadd,\uadd)
\leq \begin{cases}
\calR_2(\xdag,\udag) +\frac{\delta^o}{\alpha(\delta)}
\text{ with \eqref{apriori}}\\
\calR_2(\xdag,\udag) \text{ with \eqref{aposteriori}}\end{cases}
\end{equation}
in case of minimizers of \eqref{minTikh} with \eqref{defS} and \eqref{Tikh_aao_u}.
\item
In both cases \eqref{xaddxdag}, \eqref{xaddxdag_u}, the regularized approximations $(\xadd,\uadd)$ converge weakly subsequentially to a solution of  \eqref{Axu}, \eqref{Cuy}.
\end{itemize}
\end{proposition}
\begin{proof}
The assertion follows from the results in Section 3 of \cite{HKPS07} with $u,F(u)$ there 
defined by $(x,u),\bfF(x,u)$ here. 
Indeed almost all items of \cite[Assumption 2.1]{HKPS07} easily follow from Assumption 
\ref{ass:compactXV} \ref{ass:bfFweaklyclosed} (note that it actually suffices to assume weak sequential closedness of the operator in place of \cite[Assumption 2.1 (3), (5)]{HKPS07}).
The only exception arises in case of regularization with respect to $x$ only \eqref{Tikh_aao}, where boundedness, i.e., weak compactness, of level sets 
\[
\mathcal{M}_{\alpha}(M)=\setof{(x,u)\in\calD(\bfF)}{\frac{1}{\alpha}
\calS(\bfF(x,u),(0,\ydel))
+\calR(x,u)\leq M}
\]
\cite[Assumption 2.1 (6)]{HKPS07}, which the proofs there actually only require to hold for sufficiently small $M$, has to be shown separately. 
For obtaining this boundedness, we will make use of Assumptions \ref{ass:A_u}, \ref{ass:A_x} in that case. 
Since we have also stated convergence with the discrepancy principle, which is not treated in \cite{HKPS07} and for completeness of exposition we provide the details of the convergence part of the proof for the case of  \eqref{Tikh_aao}.

The standard argument of minimality together with \eqref{delta} for any solution $(\xdag,\udag)$ of \eqref{Axu}, \eqref{Cuy} yields the estimate
\begin{equation}\label{est1}
\calS(\bfF(\xad,\uad),(0,\ydel))
+\frac{\alpha}{r}\norm{\xad-x_0}^r
\leq \frac{1}{o}\delta^o+\frac{\alpha}{r}\norm{\xdag-x_0}^r
\end{equation}
thus, upon division by $\alpha$ and setting $\alpha=\alpha(\delta)$ (in case of \eqref{aposteriori} using the lower estimate there) we get the estimate on $\xad$ in \eqref{xaddxdag}. 
To obtain boundedness of $\uad$ as well, we use the identity 
\begin{equation}\label{idTaylor}
A(\xad,\uad)=\underbrace{A(\xdag,\udag)}_{=0}+\int_0^1 
\Bigl(A_x(x^\theta,u^\theta)(\xad-\xdag)+A_u(x^\theta,u^\theta)(\uad-\udag)\Bigr)\, d\theta
\end{equation}
where $x^\theta=\xdag+\theta(\xad-\xdag)$, $u^\theta=\udag+\theta(\uad-\udag)$
and Assumptions \ref{ass:A_u}, \ref{ass:A_x}, as well as \eqref{est1} or the upper bound in \eqref{aposteriori} 
\[
\norm{A(\xad,\uad)}_{W^*}^m \leq \bar{C}_\delta=\frac{m}{\rho}\begin{cases}
\frac{1}{o}\delta^o+\frac{\alpha(\delta)}{r}\norm{x_0-\xdag}_X^r \text{ in case of \eqref{apriori}}\\
\frac{\tau}{o}\delta^o\text{ in case of \eqref{aposteriori}}\end{cases}\,.
\]
to arrive at the estimate
\begin{equation}\label{estu}
\begin{aligned}
\norm{\uad-\udag}_V\leq&C_A \left(\bar{C}_\delta^{1/m}+\psi\left(\norm{\xdag}_X+\norm{\xad-\xdag}_X\right)
\norm{\xad-\xdag}_X \right)\\
&\times\begin{cases}
(1+\norm{\udag}_V+\norm{\uad-\udag}_V) \text{ in case (a)}\\
1 \text{ in case (b)} \end{cases}\,.
\end{aligned}
\end{equation}
By the first (already proven) part of \eqref{xaddxdag}, this directly yields a bound on $\norm{\uadd-\udag}_V$ in case (b). In case (a) we additionally use the fact that the smallness assumption on $\norm{x_0-\xdag}_X$ and the growth condition on $\psi$ allows us to achieve
\[
C_A  \left(\bar{C}_\delta^{1/m}+\psi\left(\norm{\xdag}_X+\phi\left(\norm{x_0-\xdag}_X\right)\right)\phi\left(\norm{x_0-\xdag}_X\right)\right) \leq c < 1
\]
for some constant $c$ independent of $\delta$, where 
\[
\phi(s)=\begin{cases}
s+(s^r+\frac{r\delta^o}{o\alpha(\delta)})^{1/r} \text{ in case of \eqref{apriori}}\\
2s\text{ in case of \eqref{aposteriori}}\end{cases}\,.
\]
Rearranging terms in \eqref{estu} (case (a)) we therefore get  
\[
\norm{\uadd-\udag}_V\leq \frac{c}{1-c} \left(1+\norm{\udag}_V\right) \,.
\]
The rest follows by standard arguments from the assumed continuity assumptions on $A$ and $C$.
\end{proof}

\begin{remark}
Note that in case of regularization with respect to both $x$ and $u$, \eqref{Tikh_aao_u}, we do not need Assumptions \ref{ass:A_u}, \ref{ass:A_x} and can therefore also deal with situations in which a parameter-to state map not necessarily exists.

Well-definedness of $\alpha$ according to the discrepancy principle \eqref{aposteriori} follows from \cite[Lemma 1]{KSS09} provided $X\times V$ is reflexive and strictly convex and either (i) $\bfF$ is weakly closed (i.e., Assumption \ref{ass:bfFweaklyclosed} is satisfied) and $Y$ is reflexive or (ii) $\bfF$ is weak-to-weak continuous and $\calD(\bfF)$ is weakly closed (which by the assumed convexity of this set is satisfied, e.g., if it is closed wrt. norm convergence).

If $X\times V$ satisfies the Kadets-Klee property, the results of Theorem \ref{th:convTikh} imply strong convergence and if the solution is unique then by a subsequence-subsequence argument the whole sequence converges.
\end{remark}

\subsection{Convergence rates}
As usual we can conclude convergence rates with respect to the Bregman distance under source conditions. 
Just exemplarily we will  state a rates result for the all-at-once Tikhonov method with regularization with respect to both $x$  and $u$ under a benchmark source condition.
To this end, we use an element of the subgradient
\begin{equation}\label{xizeta}
\xi\in \partial\calR(\xdag,\udag)\text{ where }\calR(x,u)=\calR_2(x,u)=\frac{1}{r}\norm{x-x_0}_X^r+\frac{1}{q}\norm{u-u_0}_V^q
\end{equation}
to define the Bregman distance 
\begin{equation}\label{Bregdist}
D_{\xi}^{(x_0,u_0)}((x,u),(\xdag,\udag))= \calR(x,u)-\calR(\xdag,\udag)-\langle \xi,(x-\xdag,x-\udag)\rangle
\end{equation}
and impose a local smoothness condition on $\bfF$.
\begin{assumption}\label{ass:Lipschitz}
\[
\begin{aligned}
&\norm{A_x(\xdag,\udag)(x-\xdag)+A_u(\xdag,\udag)(u-\udag)}
+\norm{C'(\udag)(u-\udag)}\\
&\leq
\bar{C}_L\Bigl(\frac{\rho}{m}\norm{A(x,u)-A(\xdag,\udag)}^m+\frac{1}{o}\norm{C(u)-C(\udag)}^o\Bigr)^{1/t}\\
&\qquad
+L D_{\xi}^{(x_0,u_0)}((x,u),(\xdag,\udag))\,,
\end{aligned}
\]
for some $\xi$ according to \eqref{xizeta}, $t>0$, and all $(x,u)\in\calD(\bfF)$.
\end{assumption}
In case $m=o=t$, Assumption \ref{ass:Lipschitz} follows from the inverse triangle inequality and the Taylor remainder estimate 
\[
\begin{aligned}
&\norm{A(\xdag,\udag)+A_x(\xdag,\udag)(x-\xdag)+A_u(\xdag,\udag)(u-\udag)-A(x,u)}\\
&+\norm{C(\udag)+C'(\udag)(u-\udag)-C(u)}
\leq
L D_{\xi}^{(x_0,u_0)}((x,u),(\xdag,\udag))\,,
\end{aligned}
\]
which in the quadratic Hilbert space case $D_{\xi}^{(x_0,u_0)}((x,u),(\xdag,\udag))=\frac12\norm{x-\xdag}^2+\frac12\norm{u-\udag}^2$ corresponds to the usual estimate obtained under a local Lipschitz condition on $\bfF'$.

\begin{proposition}\label{th:ratesTikh}
Let $\bfF'$ satisfy Assumption \ref{ass:Lipschitz} and let, for some $(\vmod,\vobs)\in W\times Y^*$ with $L\norm{(\vmod,\vobs)}:=L\Bigl(\norm{\vmod}_{W}+\norm{\vobs}_{V^*}\Bigr)<1$ the source condition
\begin{equation}\label{sc_aao_u}
\xi=\bfF'(\xdag,\udag)^*(\vmod,\vobs)
\end{equation}
hold for some $\xi\in\partial\calR_2(\xdag,\udag)$. 

Then for the Tikhonov minimizers according to \eqref{minTikh}, \eqref{defS}, \eqref{Tikh_aao_u} with the apriori choice
\begin{equation}\label{apriori_rate}
\alpha(\delta)\sim\delta^\frac{o}{t^*}
\end{equation} 
(with $t^*=\frac{t}{t-1}$) or the a posteriori choice \eqref{aposteriori} we have 
\begin{equation}\label{rate}
D_{\xi}^{(x_0,u_0)}((\xadd,\uadd),(\xdag,\udag))
=O(\delta^{o/t})
\end{equation}
\end{proposition}
\begin{proof}
The proof follows the lines of the classical rates proof from \cite{EKN89}.

By \eqref{xaddxdag_u} and \eqref{sc_aao_u} we have in case of \eqref{aposteriori}
\[
\begin{aligned}
&D_{\xi}^{(x_0,u_0)}((\xadd,\uadd),(\xdag,\udag))\\
&\leq -\langle (\vmod,\vobs),\bfF'(\xdag,\udag)(\xadd-\xdag,\uadd-\udag)\rangle \\
&\leq \norm{(\vmod,\vobs)}
\Bigl(
\bar{C}_L\Bigl(\frac{\rho}{m}\norm{A(\xadd,\uadd)-A(\xdag,\udag)}^m\\
&\quad+\frac{1}{o}\norm{C(\uadd)-C(\udag)}^o\Bigr)^{1/t}
+L D_{\xi}^{(x_0,u_0)}((\xadd,\uadd),(\xdag,\udag))\Bigr)
\end{aligned}
\]
hence with $c=1-L\norm{(\vmod,\vobs)}\Bigr)>0$ 
\[
cD_{\xi}^{(x_0,u_0)}((\xadd,\uadd),(\xdag,\udag))\leq 
\norm{(\vmod,\vobs)} \Bigl(\frac{\bar{C}_L(\tau+1)\delta^o}{o}\Bigr)^{1/t}\,.
\]
In case of \eqref{apriori_rate} we get from minimality (cf. \eqref{est1}) 
\[
\calS(\bfF(\xad,\uad),(0,\ydel))+\alpha\calR_2(\xad)
\leq \frac{1}{o}\delta^o+\alpha\calR_2(\xdag)\,,
\]
hence after division by $\alpha$ and by definition \eqref{Bregdist} of the Bregman distance, as well as the source condition \eqref{sc_aao_u}
\[
\begin{aligned}
&\frac{\calS(\bfF(\xad,\uad),(0,\ydel))}{\alpha}+D_{\xi}^{(x_0,u_0)}((\xad,\uad),(\xdag,\udag))\\
&\leq 
\frac{\delta^o}{\alpha}-\langle (\vmod,\vobs),\bfF'(\xdag,\udag)(\xad-\xdag,\uad-\udag)\rangle \\
&\leq
\frac{\delta^o}{\alpha}+\norm{(\vmod,\vobs)}
\Bigl(
\bar{C}_L\Bigl(\calS(\bfF(\xad,\uad),(0,\ydel))\Bigr)^{1/t}\\
&\quad +L D_{\xi}^{(x_0,u_0)}((\xad,\uad),(\xdag,\udag))\Bigr)\,,
\end{aligned}
\]
hence with $c$ as above and using Young's inequality
\[
\begin{aligned}
&\frac{\calS(\bfF(\xad,\uad),(0,\ydel))}{\alpha}+cD_{\xi}^{(x_0,u_0)}((\xad,\uad),(\xdag,\udag))\\
&\leq \frac{\delta^o}{\alpha}+\norm{(\vmod,\vobs)} \bar{C}_L\Bigl(\calS(\bfF(\xad,\uad),(0,\ydel))\Bigr)^{1/t}\\
&\leq \frac{\delta^o}{\alpha}+\Bigl(\Bigl(\frac{\norm{(\vmod,\vobs)} \bar{C}_L}{t}\Bigr)^t 2\alpha\Bigr)^{\frac{1}{t-1}}+\frac{\calS(\bfF(\xad,\uad),(0,\ydel))}{2\alpha}\,,
\end{aligned}
\]
which by the choice \eqref{apriori_rate} yields \eqref{rate}.
\end{proof}

\begin{remark}
Note that in case of regularization with respect to $x$ only, i.e., \eqref{Tikh_aao}, we do get boundedness also of the $u$ part via Assumptions \ref{ass:A_u}, \ref{ass:A_x}. However, the sharp estimate \eqref{xaddxdag_u}, that was crucially used in the rates proof above, fails to hold in general. Therefore we do not expect to get (optimal) convergence rates in that case.
\end{remark}

\subsection{Comparison of implementation}
To compare the all-at-once Tikhonov method \eqref{minTikh}, with \eqref{defS} and \eqref{Tikh_aao} or \eqref{Tikh_aao_u} with Tikhonov regularization for the reduced formulation
\begin{equation}\label{Tikh_red}
\min_{x\in\calD(F)} \frac{1}{o}\norm{F(x)-\ydel}_Y^o+\frac{\alpha}{r}\norm{x-x_0}^r
\end{equation}
we write the latter as a PDE constrained minimization problem
\begin{equation}\label{Tikh_red_PDEcons}
\min_{(x,u)\in\calD(\bfF)} \frac{1}{o}\norm{C(u)-\ydel}_Y^o+\frac{\alpha}{r}\norm{x-x_0}^r \quad \mbox{ s.t. } A(x,u)=0\,,
\end{equation}
and denote the Tikhonov minimizer as well as its corresponding state by a bar, i.e., $\xadb$ minimizes \eqref{Tikh_red_PDEcons} and $A(\xadb ,\uadb )=0$.
For convenience of exposition we will assume $W$ to be reflexive and identify it with its bidual.

We first of all show that reduced Tikhonov regularization \eqref{Tikh_red_PDEcons}, \eqref{Tikh_red} is equivalent to all-at-once Tikhonov regularization \eqref{minTikh}, with \eqref{defS} and \eqref{Tikh_aao} in case $m=1$ with $\rho$ sufficiently large.
This is due to exact penalization, cf., e.g., \cite[Theorem 17.3]{NocedalWright}, whose proof remains valid in the Banach space setting. More precisely, for this purpose $\rho$ has to be larger than the norm of the adjoint state, i.e., the solution $\bar{w}_\alpha^\delta\in W$ to 
\[
A_u(\xadb ,\uadb )^* w= C'(\uadb )^*\, J_o^Y(C(\uadb )-\ydel) \text{ in }V^*
\]
for $(\xadb ,\uadb )$ solving \eqref{Tikh_red_PDEcons}. Existence of this adjoint state is ensured by Fr\'{e}chet differentiability of $C$ together with Assumption \ref{ass:A_u}, which also implies that 
\begin{equation}\label{bdw}
\norm{\bar{w}_\alpha^\delta}_W\leq C_A \norm{C'(\uadb )}\, \norm{C(\uadb )-\ydel}_Y^{o-1}\,.
\end{equation}
Moreover, by the same arguments as in the proof of Theorem \ref{th:convTikh}, uniform (wrt. $\delta>0$) boundedness of $(\xadbd,\uadbd)$ can be concluded from its definition as a minimizer of \eqref{Tikh_red}
\[
\frac{1}{o}\frac{\norm{F(\xadbd)-\ydel}_Y^o}{\alpha(\delta)}+\frac{1}{r}\norm{\xadbd-x_0}^r
\leq\frac{1}{o}\frac{\delta^o}{\alpha(\delta)}+\frac{1}{r}\norm{\xdag-x_o}^r
\]
with $\alpha(\delta)$ chosen a priori according to \eqref{apriori} or a posteriori according to 
\[
\delta^o\leq \norm{F(\xadbd)-\ydel}^o
\leq \tau\delta^o\,,
\]
(which due to exact penalization will finally coincide with \eqref{aposteriori},)
together with the identity (cf. \eqref{idTaylor})
\begin{align*}
0&=A(\xadbd,\uadbd)=A(\xad,\uad)\\
&=\int_0^1 
\Bigl(A_x(x^\theta,u^\theta)(\xadbd-\xdag)+A_u(x^\theta,u^\theta)(\uadbd-\udag)\Bigr)\, d\theta
\end{align*}
with $x^\theta=\xdag+\theta(\xadbd-\xdag)$, $u^\theta=\udag+\theta(\uadbd-\udag)$,
which yields an estimate like \eqref{estu} with $\bar{C}_\delta=0$ and $(\xad,\uad)$ replaced by $(\xadbd,\uadbd)$.
Thus, 
from \eqref{bdw} we get a uniform bound on $\norm{\bar{w}_\alpha^\delta}_W$ and can conclude the following equivalence.
\begin{proposition}\label{prop:equivTikh}
Let Assumptions \ref{ass:A_u}, \ref{ass:A_x} be satisfied and let $C$ be Fr\'{e}chet differentiable with $C'$ mapping bounded sets to bounded sets. 
There exist $\rho>0$ sufficiently large and $\bar{\delta}>0$ sufficiently small such that for all $\delta\in(0,\bar{\delta})$ and $\alpha(\delta)$ chosen according to \eqref{apriori} or \eqref{aposteriori}, the minimizers of the all-at-once Tikhonov functional according to \eqref{minTikh} with \eqref{defS}, \eqref{Tikh_aao} with $m=1$ coincide with those of the reduced Tikhonov functional \eqref{Tikh_red}.
\end{proposition}

We now consider first order optimality conditions for the reduced and the all-at-once formulations \eqref{Tikh_red_PDEcons}, \eqref{minTikh} with general $m\geq1$, so that they are not necessarily equivalent, and for this purpose assume $A$, $C$, and the occurring norms to be continuously Fr\'{e}chet differentiable (i.e., the corresponding spaces to be uniformly smooth). 
In case $\calD(\bfF)=X\times V$, we get the first order necessary conditions 
\begin{equation}\label{KKTred_PDEcons}
\begin{aligned}
&A(\xadb ,\uadb )=0\\
&\alpha\,J_r^X(\xadb-x_0)=-A_x(\xadb ,\uadb )^* \wadb\\
&A_u(\xadb ,\uadb )^* \wadb = -C'(\uadb )^*\,J_o^Y(C(\uadb )-\ydel) 
\end{aligned}
\end{equation}
for a minimizer $(\xadb ,\uadb )$ of the reduced Tikhonov functional \eqref{Tikh_red_PDEcons} 
and 
\[
\begin{aligned}
&\alpha\,J_r^X(\xad-x_0)=-\rho A_x(\xad ,\uad )^*\,J_m^{W^*}(A(\xad,\uad))\\
&\rho A_u(\xad,\uad)^*\,J_m^{W^*}(A(\xad,\uad)) = -C'(\uadb )^*\,J_o^Y(C(\uadb )-\ydel) 
\left[-\alpha J_r^V(\uad-u_0)\right]
\end{aligned}
\]
for a minimizer $(\xad,\uad)$ of the all-at-once Tikhonov functional \eqref{minTikh} with \eqref{defS} and  \eqref{Tikh_aao} or \eqref{Tikh_aao_u}, where the term in brackets is skipped in case of \eqref{Tikh_aao}.
Upon defining $\wad$ as the solution to $A_u(\xad,\uad)^* w= -C'(\uad)^*\,J_o^Y(C(\uad)-\ydel)\left[-\alpha J_r^V(\uad-u_0)\right]$, as justified by Assumption \ref{ass:A_u},
the latter can be rewritten as
\begin{equation}\label{KKTaao}
\begin{aligned}
&\rho \,J_m^{W^*}(A(\xad,\uad))=\wad\\
&\alpha\,J_r^X(\xad-x_0)=-A_x(\xad,\uad)^* \wad\\
&A_u(\xad,\uad)^* \wad= -C'(\uad)^*\,J_o^Y(C(\uad)-\ydel) 
\left[-\alpha J_r^V(\uad-u_0)\right]
\end{aligned}
\end{equation}
so very similar to \eqref{KKTred_PDEcons}.
These first order necessary optimality conditions can be justified, e.g., by a Zowe Kurcyuscz constraint qualification \cite[Section 6.1]{TroeltzschBuch}, which in the all-at-once case is an empty condition 
and in the reduced case  
amounts to surjectivity of $A'(\xadb,\uadb)$ and is thus obviously satisfied, e.g., under Assumption \ref{ass:A_u}.

In case of additional convex constrains on the parameters $\calD(\bfF)=\calC\times V$ with some convex set $\calC$, e.g., defined by the pointwise bounds $\ul{a}$, $\ol{a}$, $\ul{c}$, $\ul{q}$ in \eqref{domA:ex1}, \eqref{domA:ex2},
the second line in \eqref{KKTred_PDEcons}, \eqref{KKTaao} changes to
\begin{equation}\label{var}
\langle \alpha\,J_r^X(\tilde{x}-x_0)+A_x(\tilde{x},\tilde{u})^* \tilde{w}, x-\tilde{x}\rangle_{X^*,X}\geq0 \text{ for all }x\in\calC
\end{equation}
for $(\tilde{x},\tilde{u},\tilde{w})=(\xadb,\uadb,\wadb)$, $(\tilde{x},\tilde{u},\tilde{w})=(\xad,\uad,\wad)$, respectively, and the Zowe Kurcyuscz constraint qualification is again always satisfied in the all-at-once case 
and amounts to 
\[
A_x(\xadb,\uadb)\calC(\xadb)+A_u(\xadb,\uadb)V=W^*
\]
with
\[
\calC(\xadb)=\setof{\gamma (x-\xadb)}{\gamma\geq0, x\in\calC}
\]
in the reduced case, which again obviously holds under Assumption \ref{ass:A_u}.

\section{The iteratively regularized Gauss-Newton method}\label{sec:IRGNM}
Throughout this section we assume $A$ and $C$ to be continuously Fr\'{e}chet differentiable on $\calD(\bfF)$ and abbreviate, analogously to \eqref{defS}
\begin{equation}\label{defSIRGNM}
\begin{aligned}
&\calS(\bfF(\tilde{x},\tilde{u})+\bfF'(\tilde{x},\tilde{u})(x-\tilde{x},u-\tilde{u}),(\ymod,\yobs))\\
&=\frac{\rho}{m}\norm{A(\tilde{x},\tilde{u})+A_x(\tilde{x},\tilde{u})(x-\tilde{x})+A_u(\tilde{x},\tilde{u})(u-\tilde{u})-\ymod}_{W^*}^m\\
&\qquad+\frac{1}{o}\norm{C(\tilde{u})+C'(\tilde{u})(u-\tilde{u})-\yobs}_Y^o\,.
\end{aligned}
\end{equation}
Given some iterate $(\xkd,\ukd)\in\calD(\bfF)$, we define the next Newton (more, precisely, iteratively regularized Gauss Newton IRGNM, cf., e.g. \cite{Baku92,BakKok04,KalNeuSch08,KH10,KSS09,HohageWerner13,Werner15}) iterate
$(\xkpd(\alpha),\ukpd(\alpha))$ as a minimizer of
\begin{equation}\label{minIRGNM}
\min_{(x,u)\in\calD(\bfF)} 
\calS(\bfF(\xkd,\ukd)+\bfF'(\xkd,\ukd)(x-\xkd,u-\ukd),(0,\ydel)) 
+\alpha\calR(x,u)
\end{equation}
with
\begin{equation}\label{IRGNM_aao}
\calR(x,u)=\calR_1(x)=\frac{1}{r}\norm{x-x_0}_X^r
\end{equation}
or 
\begin{equation}\label{IRGNM_aao_u}
\calR(x,u)=\calR_2(x,u)=\frac{1}{r}\norm{x-x_0}_V^r+\frac{1}{q}\norm{u-u_0}_V^q\,,
\end{equation}
like in the previous section, cf. \eqref{Tikh_aao}, \eqref{Tikh_aao_u}.

\subsection{Well-definedness and convergence}
As compared to Tikhonov regularization, iterative regularization methods for nonlinear problems can only be proven to converge under certain structural assumptions restricting the nonlinearity of the forward operator.
We here first of all concentrate on convergence under the simple tangential cone condition.
\begin{assumption}\label{ass:tangcone}
\[
\begin{aligned}
&\calS(\bfF(\tilde{x},\tilde{u})+\bfF'(\tilde{x},\tilde{u})(x-\tilde{x},u-\tilde{u}),\bfF(x))
\leq
c_{tc}\calS(\bfF(x,u),\bfF(\tilde{x},\tilde{u}))
\end{aligned}
\]
for some $0<c_{tc}<1$ and all $(x,u),(\tilde{x},\tilde{u})\in\calD(\bfF) \cap \calB_\varrho(x_0,u_0)$
with $\calB_\varrho(x_0,u_0)$ a closed ball with sufficiently small radius $\varrho>0$ around $(x_0,u_0)$.
\end{assumption}
\noindent
Correspondingly, we consider a posteriori choice of the stopping index $k_*=k_*(\delta)$ by the discrepancy principle
\begin{equation}\label{aposterioriIRGNM}
k_*=\min\{k\in\N \ : \ \calS(\bfF(\xkd,\ukd),(0,\ydel))\leq\frac{\tau\delta^o}{o}\}
\end{equation}
and of $\alpha_k$ for $k\leq k_*$ by the inexact Newton strategy
\begin{equation}\label{aposterioriIRGNMk}
\ul{\sigma}\leq\sigma_k(\xkpd(\alpha_k),\ukpd(\alpha_k))\leq\ol{\sigma}
\end{equation}
where
\[
\sigma_k(x,u)=
\frac{\calS(\bfF(\xkd,\ukd)+\bfF'(\xkd,\ukd)(x-\xkd,u-\ukd),(0,\ydel))}{\calS(\bfF(\xkd,\ukd),(0,\ydel))}
\]
is the ratio between the predicted and the old data misfit.
Note that by definition \eqref{aposterioriIRGNM} of $k_*$, the denominator in $\sigma_k(\xkpd(\alpha_k),\ukpd(\alpha_k))$ is nonzero for $k<k_*$ and $\delta>0$.

\begin{proposition}\label{th:convIRGNM}
Let $X\times V$ be reflexive and uniformly convex and let Assumptions 
\ref{ass:compactXV}, \ref{ass:bfFweakstrongclosed}, \ref{ass:tangcone} be satisfied with $c_{tc}$ sufficiently small
$$c_{tc}<\ul{\sigma}<\ol{\sigma}<1$$
and either (i) $\bfF'(x,u)$ be weakly closed for all $(x,u)\in\calD(\bfF)$ and $Y$ be reflexive
or (ii) $\calD(\bfF)$ be weakly closed.
In case of regularization with respect to $x$ only \eqref{IRGNM_aao}, we additionally impose Assumptions \ref{ass:A_u}, \ref{ass:A_x}.

Moreover, let $\tau$ be chosen sufficiently large so that
\begin{equation}\label{closcondIRGNM}
C_S(c_{tc}+\frac{1+C_Sc_{tc}}{\tau})\leq \ul{\sigma}
\mbox{ and }c_{tc}<\frac{1-\ol{\sigma}}{2}\,,
\end{equation}
for 
\[
C_S=\max\{\frac{\rho}{m} 2^{m-1},\frac{1}{o}2^{o-1}\}\,,
\]
let $(x_0,u_0)$ be close enough to $(\xdag,\udag)$ with respect to the Bregman distance,
and let the signal-to-noise ratio condition
$$ \tau<\frac{o\calS(\bfF(x_0,u_0),(0,\ydel))}{\delta^o} $$
hold.
\begin{itemize}
\item
Then for all $k\leq k_*(\delta)-1$ with $k_*(\delta)$ according to \eqref{aposterioriIRGNM}, the iterates
$$
(\xkpd,\ukpd):=\left\{\begin{array}{ll}
(\xkpd(\alpha_k),\ukpd(\alpha_k)), \ 
\alpha_k\mbox{ as in (\ref{aposterioriIRGNMk})} & \mbox{ if } \sigma_k(x_0,u_0)\geq\ol{\sigma}\\[0.5ex]
(x_0,u_0) &\mbox{ else }\end{array}\right.
$$
and the stopping index $k_*(\delta)$ are well-defined by \eqref{minIRGNM} with \eqref{defSIRGNM} and \eqref{IRGNM_aao} or \eqref{IRGNM_aao_u}. 
\item 
There exists $C>0$ independent of $\delta\in(0,\bar{\delta})$ such that
\begin{equation}\label{xaddxdagIRGNM}
\forall\delta\in(0,\bar{\delta})\, : \ \norm{x_{k_*(\delta)}^\delta-x_0}_X^r \leq 
\norm{\xdag-x_0}_X^r \text{ and }\norm{u_{k_*(\delta)}^\delta}_V\leq C
\end{equation}
in case of minimizers of \eqref{minIRGNM} with \eqref{defSIRGNM} and \eqref{IRGNM_aao} and 
\begin{equation}\label{xaddxdagIRGNM_u}
\forall\delta\in(0,\bar{\delta})\, : \ 
\calR_2(x_{k_*(\delta)}^\delta,u_{k_*(\delta)}^\delta)\leq\calR_2(\xdag,\udag) 
\end{equation}
in case of minimizers of \eqref{minIRGNM} with \eqref{defSIRGNM} and \eqref{IRGNM_aao_u}.
\item
In both cases \eqref{xaddxdagIRGNM}, \eqref{xaddxdagIRGNM_u}, the regularized approximations $(x_{k_*(\delta)}^\delta,u_{k_*(\delta)}^\delta)$ converges weakly subsequentially to a solution of  \eqref{Axu}, \eqref{Cuy} as $\delta\to0$.
\end{itemize}
\end{proposition}
\begin{proof}
See \cite[Theorem 3]{KSS09}.   

It only remains to prove boundedness of the $u$ part in case \eqref{IRGNM_aao}. Again we concentrate on the convergence part of the proof and use minimality and the tangential cone condition Assumption \ref{ass:tangcone} to obtain 
\[
\begin{aligned}
&\calS(\bfF(\xkd,\ukd)+\bfF'(\xkd,\ukd)(\xkpd-\xkd,\ukpd-\ukd),(0,\ydel)) 
+\frac{\alpha_k}{r}\norm{\xkpd-x_0}_X^r\\
&\leq 
\calS(\bfF(\xkd,\ukd)+\bfF'(\xkd,\ukd)(\xdag-\xkd,\udag-\ukd),(0,\ydel)) 
+\frac{\alpha_k}{r}\norm{\xdag-x_0}_X^r\\
&\leq C_S(1+C_Sc_{tc})\delta+C_Sc_{tc}
\calS(\bfF(\xkd,\ukd),(0,\ydel)) 
+\frac{\alpha_k}{r}\norm{\xdag-x_0}_X^r\,,
\end{aligned}
\]
which for $k\leq k_*-1$ by \eqref{aposterioriIRGNM}, \eqref{aposterioriIRGNMk}, yields
\[
\begin{aligned}
&\Bigl(\ul{\sigma}-C_S(c_{tc}+\frac{1+C_Sc_{tc}}{\tau})\Bigr) \calS(\bfF(\xkd,\ukd),(0,\ydel)) 
+\frac{\alpha_k}{r}\norm{\xkpd-x_0}_V^r\\
&\leq \frac{\alpha_k}{r}\norm{\xdag-x_0}_V^r\,,
\end{aligned}
\]
which upon division by $\alpha_k$ gives a uniform bound on $\norm{\xkpd-x_0}_V$ (namely the one stated in \eqref{xaddxdagIRGNM}; obviously, \eqref{xaddxdagIRGNM_u} can be obtained in the same manner). 
Thus up to the positive factor $\Bigl(\ul{\sigma}-C_S(c_{tc}+\frac{1+C_Sc_{tc}}{\tau})\Bigr)$ we are in the same situation as in \eqref{est1} and thus can use the same arguments to prove uniform boundedness of $u_{k_*}^\delta$. As a matter of fact, since we only consider the a posteriori regularization parameter choice now, the situation is even simpler, since we can use the upper estimate in \eqref{aposterioriIRGNM} in the identity \eqref{idTaylor} (with $(\xad,\uad)$ replaced by $(\xkd,\ukd)$).
\end{proof}

A sufficient condition for Assumption \ref{ass:tangcone} to hold is, like in the reduced case, the adjoint range invariance condition
\begin{equation}\label{eq:ari}
\bfF'(\tilde{x},\tilde{u})=\bfR_{(x,u)}^{(\tilde{x},\tilde{u})} \bfF'(x,u) 
\mbox{ with } \|\bfR_{(x,u)}^{(\tilde{x},\tilde{u})}-I\|\leq c_R<1
\end{equation}
for some $0<c_{tc}<1$ and all $(x,u),(\tilde{x},\tilde{u})\in\calD(\bfF) \cap \calB_\varrho(x_0,u_0)$, as well as linear operators $\bfR_{(x,u)}^{(\tilde{x},\tilde{u})}:W^*\times Y\to W^*\times Y$. 
More explicitly, with $\bfR_{(x,u)}^{(\tilde{x},\tilde{u})}=\left(\begin{array}{cc}R_{11}&R_{12}\\R_{21}&R_{22}\end{array}\right)$ the identity between the derivatives of $\bfF$ can be rewritten as
\begin{equation}\label{eq:adjrangeinvar}
\begin{aligned}
A_x(\tilde{x},\tilde{u})&=R_{11}A_x(x,u)\\
A_u(\tilde{x},\tilde{u})&=R_{11}A_u(x,u)+R_{12} C\\
0&= R_{21}A_x(x,u)\\
C'(\tilde{u})&= R_{21}A_u(x,u)+R_{22}C'(u)
\end{aligned}
\end{equation}
Actually, this is the condition that we will check for some of our examples.

Closely related is the following range invariance condition
\begin{assumption}\label{ass:rangeinvar} 
\begin{equation}\label{eq:ri}
\bfF'(\tilde{x},\tilde{u})= \bfF'(x,u) \bfR_{(x,u)}^{(\tilde{x},\tilde{u})} 
\mbox{ with } \|\bfR_{(x,u)}^{(\tilde{x},\tilde{u})}-I\|\leq c_R<1
\end{equation}
for some $0<c_{tc}<1$ and all $(x,u),(\tilde{x},\tilde{u})\in\calD(\bfF) \cap \calB_\varrho(x_0,u_0)$, as well as linear operators $\bfR_{(x,u)}^{(\tilde{x},\tilde{u})}:X\times V\to X\times V$.
\end{assumption} 
Again we can write the identity between the derivatives of $\bfF$ more explicitly by using the block decomposition $\bfR_{(x,u)}^{(\tilde{x},\tilde{u})}=\left(\begin{array}{cc}R_{11}&R_{12}\\R_{21}&R_{22}\end{array}\right)$:
\begin{equation}\label{eq:rangeinvar}
\begin{aligned}
A_x(\tilde{x},\tilde{u})&=A_x(x,u)R_{11}+A_u(x,u)R_{21}\\
A_u(\tilde{x},\tilde{u})&=A_x(x,u)R_{12}+A_u(x,u)R_{22}\\
0&= C'(u)R_{21}\\
C'(\tilde{u})&= C'(u)R_{22}
\end{aligned}
\end{equation}
Under this alternative condition we get convergence with an a priori choice of $\alpha_k$ and $k_*$, however only in Hilbert space. Note that convergence under such a range invariance condition in the general Banach space setting is still an open problem also for the conventional reduced formulation.
We quote a convergence result for the case of \eqref{IRGNM_aao_u}.

\begin{proposition}\label{th:convIRGNM_range}
Let $o=m=r=2$ and $X\times V$ and $W^*\times V$ be Hilbert spaces and let Assumptions 
\ref{ass:compactXV}, \ref{ass:bfFweakstrongclosed}, \ref{ass:rangeinvar} be satisfied with $c_{tc}$ sufficiently small.
Moreover, let $\tau$ be chosen sufficiently large, and let $(x_0,u_0)$ be close enough to $(\xdag,\udag)$.

Then for all $k\in\N$, the iterates
$(\xkpd,\ukpd)=
(\xkpd(\alpha_k),\ukpd(\alpha_k))$ 
with an apriori chosen sequence $(\alpha_k)_{k\in\N}$ satisfying 
\begin{equation}\label{alpha_apriori}
\alpha_k\to0\mbox{ as }k\to\infty\mbox{ and }1\leq\frac{\alpha_k}{\alpha_{k+1}}\leq C_\alpha
\end{equation}
for some $C_\alpha\geq1$ are well-defined by \eqref{minIRGNM} with \eqref{defSIRGNM} and 
\eqref{IRGNM_aao_u}, and with a choice of $k_*(\delta)$ such that 
\begin{equation}\label{kastaposteriori}
k_*(\delta)\to\infty\mbox{ and  }\frac{\delta}{\sqrt{\alpha_{k_*(\delta)}}}\to0\mbox{ as }\delta\to0
\end{equation} 
$(x_{k_*(\delta)}^\delta,u_{k_*(\delta)}^\delta)$ converges strongly to a solution of  \eqref{Axu}, \eqref{Cuy} as $\delta\to0$.
\end{proposition}
\begin{proof}
See \cite[Theorem 2.7]{dpapertheo}. 
\end{proof}

\subsection{Convergence rates}
Exactly as in the proof of the first part (case of a posteriori regularization parameter choice) of Proposition \ref{th:ratesTikh}, using \eqref{xaddxdagIRGNM_u} we get rates under a source condition (cf. \cite{KH10,Werner15} for the reduced case).
\begin{proposition}\label{th:ratesIRGNM}
Let $\bfF'$ satisfy Assumption \ref{ass:Lipschitz} and let, for and some $(\vmod,\vobs)\in W\times Y^*$ with $L\norm{(\vmod,\vobs)}<1$ the source condition
\eqref{sc_aao_u}
hold.

Then for the Newton iterates according to \eqref{IRGNM_aao_u} with the a posteriori choice \eqref{aposterioriIRGNM}, \eqref{aposterioriIRGNMk} or the a priori choice \eqref{alpha_apriori}, 
$\alpha_{k^*(\delta)}\sim\delta^{\frac{o}{t^*}}$ (cf. \eqref{apriori_rate}),  we have  
\[
D_{\xi}^{(x_0,u_0)}((x_{k_*(\delta)}^\delta,u_{k_*(\delta)}^\delta),(\xdag,\udag))
=O(\delta^{o/t})
\] 
(cf. \eqref{rate}).
\end{proposition}

\subsection{Comparison of implementation}
We compare one step of the the all-at-once IRGNM to one step of the reduced IRGNM 
\begin{equation}\label{IRGNM_red}
\min_{x\in\calD(F)}
\frac{1}{o}\norm{F(\xkdb)+F'(\xkdb)(x-\xkdb)-\ydel}_Y^o 
+\frac{\alpha}{r}\norm{x-x_0}_V^r\,,
\end{equation}
where $\tilde{\calD}=\setof{(x,u,\tilde{u})}{(\xkdb,\tilde{u})\in\calD(\bfF),\ x\in X\,, u\in V}$,
and denote the iterates resulting from the reduced version by a bar.
This can be rewritten as a PDE constrained minimization problem (setting $\tilde{u}=S(\xkdb)$, $u=\tilde{u}+S'(\xkdb)(x-\xkdb)=\tilde{u}-A_u(\xkdb,\tilde{u})^{-1}A_x(\xkdb,\tilde{u})(x-\xkdb)$
\begin{equation}\label{IRGNM_red_PDEcons}
\begin{aligned}
\min_{(x,u,\tilde{u})\in\tilde{\calD}} &\frac{1}{o}\norm{C(\tilde{u})+C'(\tilde{u})(u-\tilde{u})-\ydel}_Y^o+\frac{\alpha}{r}\norm{x-x_0}^r \\ 
\mbox{ s.t. }& \begin{cases} A(\xkdb,\tilde{u})=0\,, \\
A(\xkdb,\tilde{u})+A_x(\xkdb,\tilde{u})(x-\xkdb)+A_u(\xkdb,\tilde{u})(u-\tilde{u})=0\,,
\end{cases}
\end{aligned}
\end{equation}
in particular it contains the nonlinear model $A(\xkdb,\tilde{u})=0$ as well as the linearized one as constraints.

Again we assume $W$ to be reflexive and start with the exact penalization case $m=1$ in \eqref{minIRGNM}.
However, the PDE constrained minimization problem to which \eqref{minIRGNM} with $m=1$ and $\rho$ sufficiently large is equivalent (analogously to Proposition \ref{prop:equivTikh}) is the following
\begin{equation}\label{IRGNM_aao_PDEcons}
\begin{aligned}
\min_{(x,u,\tilde{u})\in\tilde{\calD}} &\frac{1}{o}\norm{C(\ukd)+C'(\ukd)(u-\ukd)-\ydel}_Y^o+\frac{\alpha}{r}\norm{x-x_0}^r \\
\mbox{ s.t. }&
A(\xkd,\ukd)+A_x(\xkd,\ukd)(x-\xkd)+A_u(\xkd,\ukd)(u-\ukd)=0
\end{aligned}
\end{equation}
so only the linearized model appears as a constraint here, and the problem is obviously not the same as the reduced one \eqref{IRGNM_red_PDEcons}, differently from Section \ref{sec:Tikh}.
For the sake of completeness, we point out that exact penalization indeed holds for sufficiently large but finite $\rho$ since the adjoint state $w_k$ for \eqref{IRGNM_aao_PDEcons} by Assumption \ref{ass:A_u} and Proposition \ref{th:convIRGNM} is well defined by 
\[
A_u(\xkd,\ukd)^* w_k= -C'(\ukd)^*J_o^Y(C(\ukd)+C'(\ukd)(\ukpd-\ukd)-\ydel) \text{ in }V^*
\]
and uniformly (wrt $\delta\in(0,\bar{\delta})$ and $k\leq k_*(\delta)$) bounded, 
thus it is possible to choose such a uniform penalty parameter $\rho\geq \sup_{\delta\in(0,\bar{\delta})}\max_{k\in\{1,\ldots,k_*(\delta)\}} \norm{w_k}_W$.

\medskip

The first order optimality conditions for the reduced and the all-at-once formulations \eqref{IRGNM_red_PDEcons}, \eqref{minIRGNM}, \eqref{IRGNM_aao_PDEcons} in case $\calD(\bfF)=X\times V$ read as follows:
\begin{equation}\label{KKTred_PDEconsIRGNM}
\begin{aligned}
&A(\xkdb ,\ukdb )=0\\
&A(\xkdb,\ukdb )+A_x(\xkdb,\ukdb )(\xkpdb-\xkdb)+A_u(\xkdb,\ukdb )(\ukpdb-\ukdb )=0\\
&\alpha\,J_r^X(\xkpdb-x_0)=-A_x(\xkdb,\ukdb )^* \wkpdb\\
&A_u(\xkdb,\ukdb )^* \zkdb= -(C''(\ukdb)(\ukpdb-\ukdb))^*\,J_o^Y(C(\ukdb)+C'(\ukdb)(\ukpdb-\ukdb)-\ydel)\\
&\hspace*{3cm}
-(A_{ux}(\xkdb,\ukdb)(\xkpdb-\xkdb)+A_{uu}(\xkdb,\ukdb)(\ukpdb-\ukdb))^*\wkpdb\\
&A_u(\xkdb,\ukdb)^* \wkpdb = -C'(\ukdb)^*\,J_o^Y(C(\ukdb)+C'(\ukdb)(\ukpdb-\ukdb)-\ydel)
\end{aligned}
\end{equation}
for a minimizer $(\xkpdb,\ukdb,\ukpdb)$ of the reduced version \eqref{IRGNM_red_PDEcons},
(where the equation for the Lagrange multiplier $\zkdb$ corresponding to the second PDE constraint may be skipped)
\begin{equation}\label{KKTaaoIRGNM}
\begin{aligned}
&\rho \,J_m^{W^*}(A(\xkd,\ukd)+A_x(\xkd,\ukd)(\xkpd-\xkd)+A_u(\xkd,\ukd)(\ukpd-\ukd))=\wkpd\\
&\alpha\,J_r^X(\xkpd-x_0)=-A_x(\xkd,\ukd)^* \wkpd\\
&A_u(\xkd,\ukd)^* \wkpd= -C'(\ukd)^*\,J_o^Y(C(\ukd)+C'(\ukd)(\ukpd-\ukd)-\ydel)
\end{aligned}
\end{equation}
for a minimizer $(\xkpd,\ukpd)$ of the all-at-once version \eqref{minIRGNM}, 
and 
\begin{equation}\label{KKTaao_PDEconsIRGNM}
\begin{aligned}
&A(\xkd,\ukd)+A_x(\xkd,\ukd)(\xkpd-\xkd)+A_u(\xkd,\ukd)(\ukpd-\ukd))=0\\
&\alpha\,J_r^X(\xkpd-x_0)=-A_x(\xkd,\ukd)^* \wkpd\\
&A_u(\xkd,\ukd)^* \wkpd= -C'(\ukd)^*\,J_o^Y(C(\ukd)+C'(\ukd)(\ukpd-\ukd)-\ydel)
\end{aligned}
\end{equation}
for a minimizer $(\xkpd,\ukpd)$ of the all-at-once version \eqref{minIRGNM} in case $m=1$ and $\rho$ sufficiently large, i.e., of \eqref{IRGNM_aao_PDEcons}.

Justification of these first order optimality conditions by Fr\'{e}chet differentiability of $A$ and $C$ and a Zowe Kurcyuscz constraint qualification is again trivial in case of \eqref{minIRGNM} with $m>1$.  
In case of \eqref{IRGNM_red_PDEcons}, we need surjectivity of 
\[
G'(x,u,\tilde{u})=\left(\begin{matrix}
0&0&A_u(\xkdb,\tilde{u})\\ 
A_x(\xkdb,\tilde{u})&A_u(\xkdb,\tilde{u})&A_{xu}(\xkdb,\tilde{u})(x-\xkdb)+A_{xu}(\xkdb,\tilde{u})(u-\tilde{u})\end{matrix}\right)
\] 
at $(x,u,\tilde{u})=(\xkpdb,\ukpdb,\ukdb)$
and in case of \eqref{IRGNM_aao_PDEcons} surjectivity of 
\[
G'(x,u)=\left(\begin{matrix} A_x(\xkd,\ukd)&A_u(\xkd,\ukd)\end{matrix}\right)
\] 
at $(x,u)=(\xkpd,\ukpd)$, so that in both cases the constraint qualifications are satisfied under Assumption \ref{ass:A_u}.

In case of additional convex constrains on the parameters $\calD(\bfF)=\calC\times V$ with some convex set $\calC$, the second lines in \eqref{KKTred_PDEconsIRGNM}, \eqref{KKTaaoIRGNM}, \eqref{KKTred_PDEconsIRGNM} again change to variational inequalities of the form \eqref{var}.

\section{Landweber iteration}\label{sec:LW}
Gradient methods applied to the unregularized least squares problem (i.e., \eqref{Tikh_red} or \eqref{minTikh} with $\alpha=0$) lead to the Landweber iteration in a reduced
\begin{equation}\label{LW_red}
\begin{aligned}
&\xskpdb=\xskdb-\mu_k F'(\xkdb)^*\, J_o^Y(F(\xkdb)-\ydel)\\
&\xkpdb=J_{r^*}^{X^*}(\xskpdb)
\end{aligned}
\end{equation}
and an all-at-once setting
\begin{equation}\label{LW_aao}
\begin{aligned}
&(\xskpd,\uskpd)=(\xskd,\uskd)-\mu_k \bfF'(\xkd,\ukd)^*\, (J_m^{W^*},J_o^Y)(\bfF(\xkd,\ukd)-(0,\ydel))\\
&(\xkpd,\ukpd)=(J_{r^*}^{X^*}(\xskpd),J_{q^*}^{V^*}(\uskpd))
\end{aligned}
\end{equation}
with appropriately chosen stepsizes $\mu_k$, cf., e.g., \cite{KSS09}.
We here assume that $X$ and $V$ are smooth and $s$-convex so that the duality mappings $J_{r^*}^{X^*}$, $J_{q^*}^{V^*}$ are single valued and in fact inverses of $J_{r}^{X}$, $J_{q}^{V}$, respectively. Moreover, we postulate continuous Fr\'{e}chet differentiability of $A$ and $C$, and, for \eqref{LW_red}, also Assumption \ref{ass:A_u} to guarantee well-definedness of $F$. 

Under a tangential cone condition 
\[
\begin{aligned}
\norm{F(\tilde{x})+F'(\tilde{x})(x-\tilde{x})-F(x)}_Y
\leq
c_{tc}\norm{F(\tilde{x})-F(x)}_Y
\end{aligned}
\]
for some $0<c_{tc}<1$ and all $x\in\calB_\varrho(x_0)$
in the reduced case \eqref{LW_red}
or Assumption \ref{ass:tangcone} in the all-at-once case \eqref{LW_aao}, we have convergence of the iterates $\bar{x}_k^0$ or $(x_k^0,u_k^0)$ to a solution of \eqref{Axu}, \eqref{Cuy} as $k\to\infty$ in the exact case $\delta=0$. Under the same conditions, with additionally $Y$ or $W^*\times Y$ being uniformly smooth, an a posteriori stopping rule according to 
\[
k_*=\min\{k\in\N \ : \ \norm{F(\xkdb),\ydel)}_Y
\leq \tau \delta\} \,,
\]
for \eqref{LW_red} 
or \eqref{aposterioriIRGNM} for \eqref{LW_aao}
yields convergence of $\bar{x}_{k_*(\delta)}^\delta$ or $(x_{k_*(\delta)}^\delta,u_{k_*(\delta)}^\delta)$ to a solution of \eqref{Fxy} or \eqref{Axu}, \eqref{Cuy} as $\delta\to0$ holds.
This is an immediate consequence of Theorems 1 and 2 in \cite{KSS09}.
Note that even if the all-at-once tangential cone condition Assumption \ref{ass:tangcone} might be hard to verify for nonlinear problems, Landweber still makes sense for linear problems where this condition is trivially satisfied. And it actually makes a big difference to the reduced version also in the linear case as we will see in the following. 

\subsection{Comparison of implementation}
Although we do not deal with minimization problems now, each reduced Landweber step can be split into a system that resembles \eqref{KKTred_PDEcons}, but is has triangular (so not fully coupled) structure.
Namely, using the definition of $F=C\circ S$, the chain rule, and the Implicit Function Theorem for differentiating $S$
\[
A(x,S(x))=0\,,\quad A_x(x,S(x))+A_u(x,S(x))S'(x)=0\,,
\] 
i.e., 
\[  
F'(x)^*\, J_o^Y(F(x)-\ydel)=-A_x(x,S(x))^*(A_u(x,S(x))^*)^{-1}C'(S(x))^*\, J_o^Y(C(S(x))-\ydel)
\]
and setting $\ukdb=S(\xkdb)$, $\wkdb=-(A_u(\xkdb,\ukdb)^*)^{-1}C'(\ukdb)^*\, J_o^Y(C(\ukdb)-\ydel)$ we can rewrite one reduced Landweber step as 
\[
\begin{aligned}
&A(\xkdb,\ukdb)=0\\
&A_u(\xkdb,\ukdb)^*\wkdb=-C'(\ukdb)^*\, J_o^Y(C(\ukdb)-\ydel)\\
&\xskpdb=\xskdb-\mu_k A_x(\xkdb,\ukdb)^*\wkdb\\
&\xkpdb=J_{r^*}^{X^*}(\xskpdb)
\end{aligned}
\]
which involves solution of a nonlinear and a linearized (adjoint) model,
whereas the all-at-once version using
\[
\bfF'(x,u)^*(J_m^{W^*},J_o^Y)(\bfF(x,u)-(0,\ydel))=
\left(\begin{matrix}
A_x(x,u)&A_u(x,u)\\
0&C'(u)
\end{matrix}\right)^*
\left(\begin{matrix}
J_m^{W^*}(A(x,u))\\
J_o^Y(C(u)-\ydel))
\end{matrix}\right)
\]
reads
\[
\begin{aligned}
&\xskpd=\xskd-\mu_k A_x(\xskd,\uskd)^*J_m^{W^*}(A(\xskd,\uskd))\\
&\uskpd=\uskd-\mu_k \Bigl(A_u(\xskd,\uskd)^*J_m^{W^*}(A(\xskd,\uskd))+C'(u)^*J_o^Y(C(u)-\ydel)\Bigr)\\
&(\xkpd,\ukpd)=(J_{r^*}^{X^*}(\xskpd),J_{q^*}^{V^*}(\uskpd))\,,
\end{aligned}
\]
(where again we have assumed $W$ to be reflexive).
So as opposed to all other schemes considered here and to the reduced Landweber setting, in an all-at-once Landweber step no linear or nonlinear model is solved.  
 
\section{Application and Numerical Tests}\label{sec:numex}
Since the main difference in implementation seems to arise in the context of nonlinear PDEs (at least as far as the IRGNM is concerned) we consider an inverse source problem for a nonlinear PDE, namely, identification of $b$ in 
\begin{equation}\label{nlbex}
-\Delta u+\xi u^3=b \mbox{ in }\Omega\,, \quad u=0\mbox{ on }\partial\Omega
\end{equation}
from measurements of $u$ in $\Omega\subseteq\R^d$, $d\in\{1,2,3\}$.
Here  $\xi\in\R$ is a known parameter, which we use in order to study different strengths of nonlinearity of the PDE and also the non-elliptic case arising when $\xi$ is smaller than the negative of the embedding constant $H_0^1(\Omega)\to L^4(\Omega)$, so that the parameter-to-state map fails to be well-defined and therefore the reduced approach is not applicable.
This example has already been used in \cite{KKV14b} to compare reduced Tikhonov regularization with an all-at-once version of the IRGNM, both with adaptive discretization.

\subsection{Verification of convergence conditions}
Indeed, even a more general version of this example satisfies the structural conditions on the forward operator \eqref{eq:adjrangeinvar}, \eqref{eq:rangeinvar} for the all-at-once IRGNM, and also their reduced versions.

\begin{lemma}\label{lem:nlsc}
Let $A$, $C$  in \eqref{Axu}, \eqref{Cuy} have the form $A(x,u)=\hat{A}(u)+Lx$, $C(u)=Cu$ with $L:X\to W^*$, $C:V\to Y$ linear operators, $\hat{A}:V\to W^*$ possibly nonlinear and Fr\'{e}chet differentiable, and let $C^\dagger$, $L^\dagger$ denote the Moore Penrose generalized inverse of $C$ and $L$, respectively.
\begin{itemize}
\item[(a)] If $\mathcal{N}(\hat{A}'(\tilde{u})-\hat{A}'(u))^\bot\subseteq \mathcal{N}(C)^\bot$, then \eqref{eq:adjrangeinvar} is satisfied with 
\[ \|\bfR_{(x,u)}^{(\tilde{x},\tilde{u})}-I\|= \|(\hat{A}'(\tilde{u})-\hat{A}'(u))C^\dagger\| \]
\item[(b)] If $\mathcal{R}(\hat{A}'(\tilde{u})-\hat{A}'(u))\subseteq \overline{\mathcal{R}(L)}$, then \eqref{eq:rangeinvar} is satisfied with 
\[ \|\bfR_{(x,u)}^{(\tilde{x},\tilde{u})}-I\|= \|L^\dagger(\hat{A}'(\tilde{u})-\hat{A}'(u))\| \]
\item[(c)] If $\hat{A}'(u)$ is boundedly invertible, $\hat{A}'$ is continuous, and $\mathcal{N}(C)=\{0\}$, then for all $\tilde{x}$ in a sufficiently small neighborhood of $x$, the reduced adjoint range invariance condition 
\begin{equation}\label{eq:adjrangeinvar_red}
F'(\tilde{x})=R_{x}^{\tilde{x}} F'(x)
\end{equation}
is satisfied with 
\[ \|R_{x}^{\tilde{x}}-I\|= \sqrt{\|C ( \hat{A}'(S(x))-\hat{A}'(S(\tilde{x})) )\hat{A}'(S(\tilde{x}))^{-1}C^\dagger\|^2+\|\mbox{Proj}_{\mathcal{R}(C)^\bot}\|^2} \]
\item[(d)] If $\hat{A}'(u)$ is boundedly invertible, $\hat{A}'$ is continuous, and $\overline{\mathcal{R}(L)}=W^*$, then the reduced range invariance condition 
\begin{equation}\label{eq:rangeinvar_red}
F'(\tilde{x})=F'(x) R_{x}^{\tilde{x}} 
\end{equation}
is satisfied with 
\[ \|R_{x}^{\tilde{x}}-I\|= \sqrt{\|L^\dagger ( \hat{A}'(S(x))-\hat{A}'(S(\tilde{x})) )\hat{A}'(S(\tilde{x}))^{-1} L\|^2+\|\mbox{Proj}_{\mathcal{N}(L)}\|^2} 
\]
\end{itemize}
\end{lemma}
\begin{proof}
Since the derivatives here simplify to $A_x(x,u)=L$, $A_u(x,u)=\hat{A}'(u)$, $C'(u)=C$, the assertions can be readily checked by setting
\[
\begin{aligned}
&(a)\quad R_{11}=I\,, \quad R_{12}=(\hat{A}'(\tilde{u})-\hat{A}'(u))C^\dagger\,, \quad R_{21}=0\,, \quad R_{22}=I\\
&(b)\quad R_{11}=I\,, \quad R_{12}=L^\dagger(\hat{A}'(\tilde{u})-\hat{A}'(u))\,, \quad R_{21}=0\,, \quad R_{22}=I\\
&(d)\quad R_{x}^{\tilde{x}}=C\hat{A}'(S(x))\hat{A}'(S(\tilde{x}))^{-1}C^\dagger\\
&(c)\quad R_{x}^{\tilde{x}}=L^\dagger\hat{A}'(S(x))\hat{A}'(S(\tilde{x}))^{-1}L
\end{aligned}
\]
\end{proof}

The above example \eqref{nlbex} fits into this framework with 
\[
\begin{aligned}
&X=L^2(\Omega)\,, \ V=H^2(\Omega)\cap H_0^1(\Omega)\,, \ W=L^2(\Omega)\,, \ Y= L^2(\Omega)\\
&\hat{A}(u)=-\Delta u+\zeta u^3\,, \ L=-\mbox{id}
\end{aligned}
\]
and therefore satisfies \eqref{eq:rangeinvar}, and in the elliptic case $\xi>\ul{\xi}=-\|\mbox{id}_{H_0^1(\Omega)\to L^4(\Omega)}\|$ also \eqref{eq:rangeinvar_red}, provided $C$ is linear. In case of full measurements relative to the model $A$, i.e., if $\mathcal{N}(C)\subseteq\mathcal{N}(\hat{A}'(u))$ for all $(x,u)\in\calD(\bfF) \cap \calB_\varrho(x_0,u_0)$, also \eqref{eq:adjrangeinvar} holds and if $\xi>\ul{\xi}$ and even $\mathcal{N}(C)=\{0\}$, $\overline{\mathcal{R}(C)}=Y$ hold, then \eqref{eq:adjrangeinvar_red} is satisfied.

\subsection{Numerical experiments}
In the following, we show numerical results for this example in the quadratic (i.e., $m=o=r=2$) Hilbert space setting in the one-dimensional situation $\Omega=(0,1)$ with a simple finite difference discretization on an equidistant grid of size $0.01$. Computational results in 2-d with an adaptive discretization can be found in \cite{KKV14b}.

Table \ref{tab:compIRGNM} shows a comparison of the reduced and the all-at-once versions of the IRGNM for different values of the nonlinearity parameter $\xi$ with one per cent noise in the data.  Here $\mbox{it}_{\mbox{\footnotesize aao}}$, $\mbox{it}_{\mbox{\footnotesize red}}$, $\mbox{cpu}_{\mbox{\footnotesize aao}}$, $\mbox{cpu}_{\mbox{\footnotesize red}}$ denote the number of iterations and CPU times (in seconds) for the all-at-once and the reduced version, respectively. The corresponding reconstructions are displayed in Figure \ref{fig:compIRGNM}.
In both cases the sequence of regularization parameters was chosen as $\alpha_k=10* 0.7^{k}$ and the constant functions with value zero were used as starting values for $b$ (and $u$). As expected, for larger values of $\xi$, i.e., stronger nonlinearity, the all-at-once version performs better. Also for negative values of $\xi$ with modulus larger than the reciprocal norm of the embedding $H_0^1(\Omega)\to L^4(\Omega)$, for which the parameter-to state-map does not exist (and the reduced IRGNM actually fails) we still get  a reasonable behavior of  the all-at-once IRGNM.

\begin{table}
\begin{tabular}{|l|l|l|l|l|l|l|l|}
\hline
$\xi$&$\tau^2$&$\mbox{it}_{\mbox{\footnotesize aao}}$&$\mbox{it}_{\mbox{\footnotesize red}}$& $\mbox{cpu}_{\mbox{\footnotesize aao}}$& $\mbox{cpu}_{\mbox{\footnotesize red}}$& 
$\frac{\|b_{k_*(\delta),\mbox{\footnotesize aao}}^\delta-b^\dagger\|_X}{\|b^\dagger\|_X}$ 
&
$\frac{\|b_{k_*(\delta),\mbox{\footnotesize red}}^\delta-b^\dagger\|_X}{\|b^\dagger\|_X}$ 
\\
\hline
   0 &4 & 34 & 32 & 0.14 & 0.10 & 0.0149 & 0.0151\\
  10 &20& 43 & 43 & 0.20 & 0.55 & 0.0996 & 0.1505\\
 100 &20& 55 & 56 & 0.28 & 0.82 & 0.0721 & 0.0770\\
1000 &20& 68 & 68 & 0.42 & 1.07 & 0.0543 & 0.0588\\
-0.5 &4 & 33 & 32 & 0.13 & 0.35 & 0.1174 & 0.2165\\
-1.  &4 & 35 & - & 0.23 & - & 0.2023 & -\\
-10  &4 & 44 & - & 0.23 & - & 0.0768 & -\\
-100 &4 & 77 & - & 0.59 & - & 0.2246 & -\\
-1000&4 & 70 & - & 0.49 & - & 0.0321 & -\\
\hline
\end{tabular}
\caption{Comparison of reduced and all-at-once IRGNM \label{tab:compIRGNM}}
\end{table}

\begin{figure}
\begin{tabular}{ccc}
$\xi=1000$ &  $\xi=-0.5$ & $\xi=-1000$\\
\includegraphics[width=0.3\textwidth]{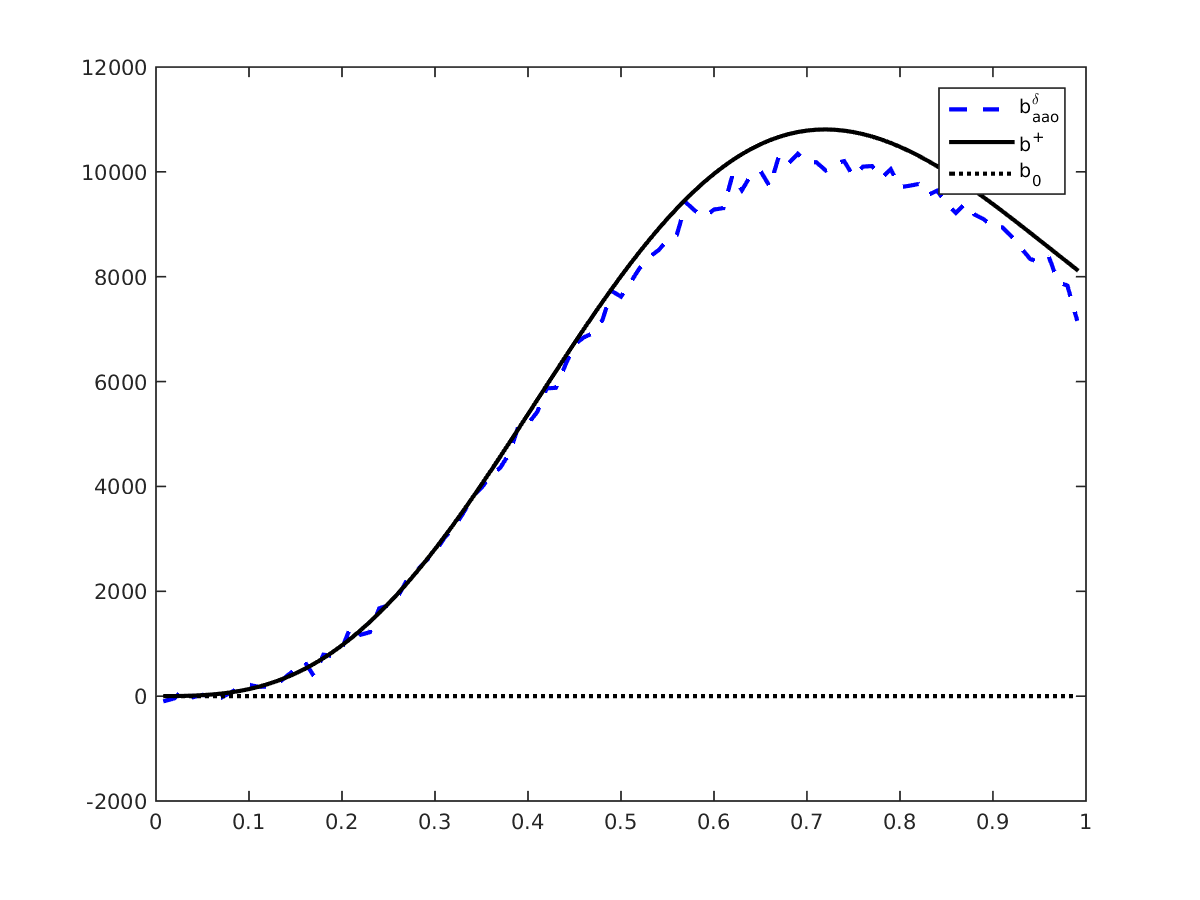}&
\includegraphics[width=0.3\textwidth]{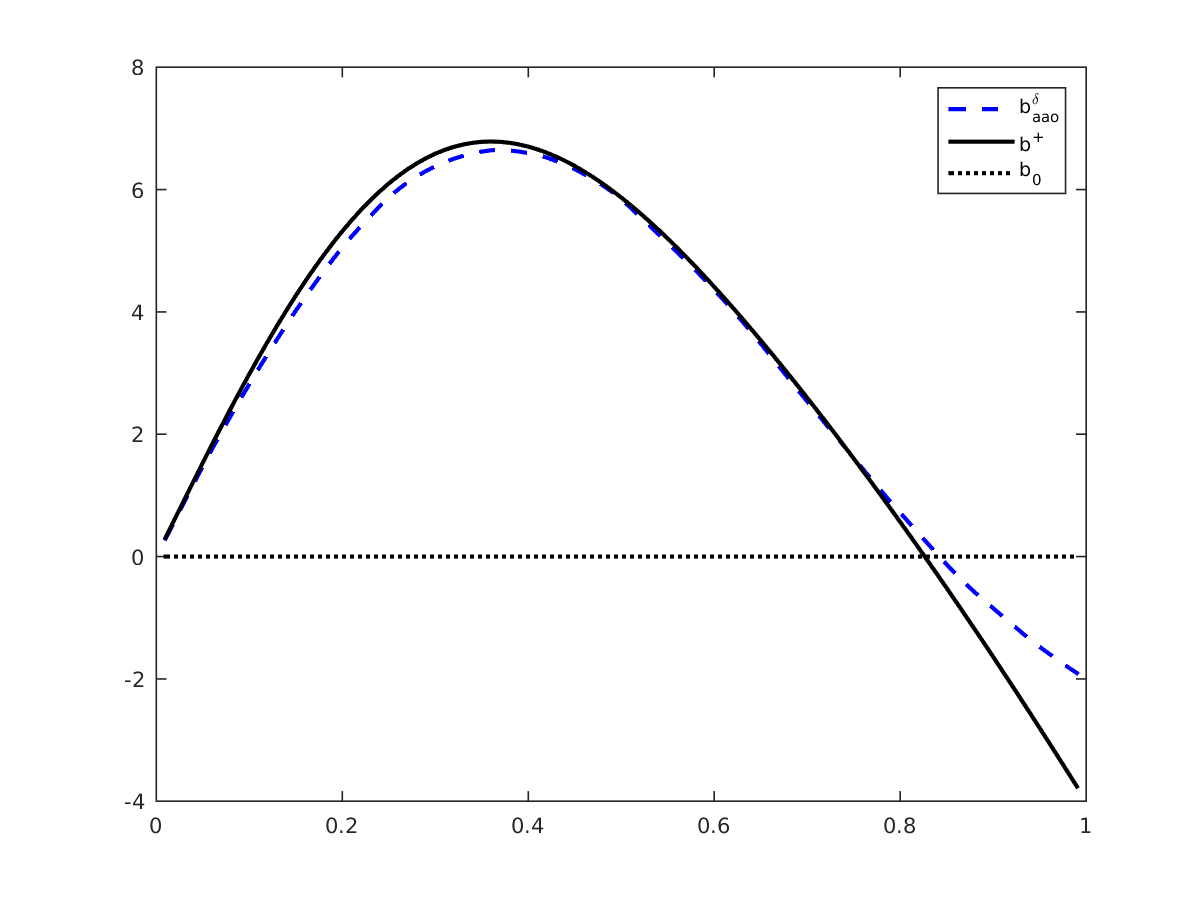}&
\includegraphics[width=0.3\textwidth]{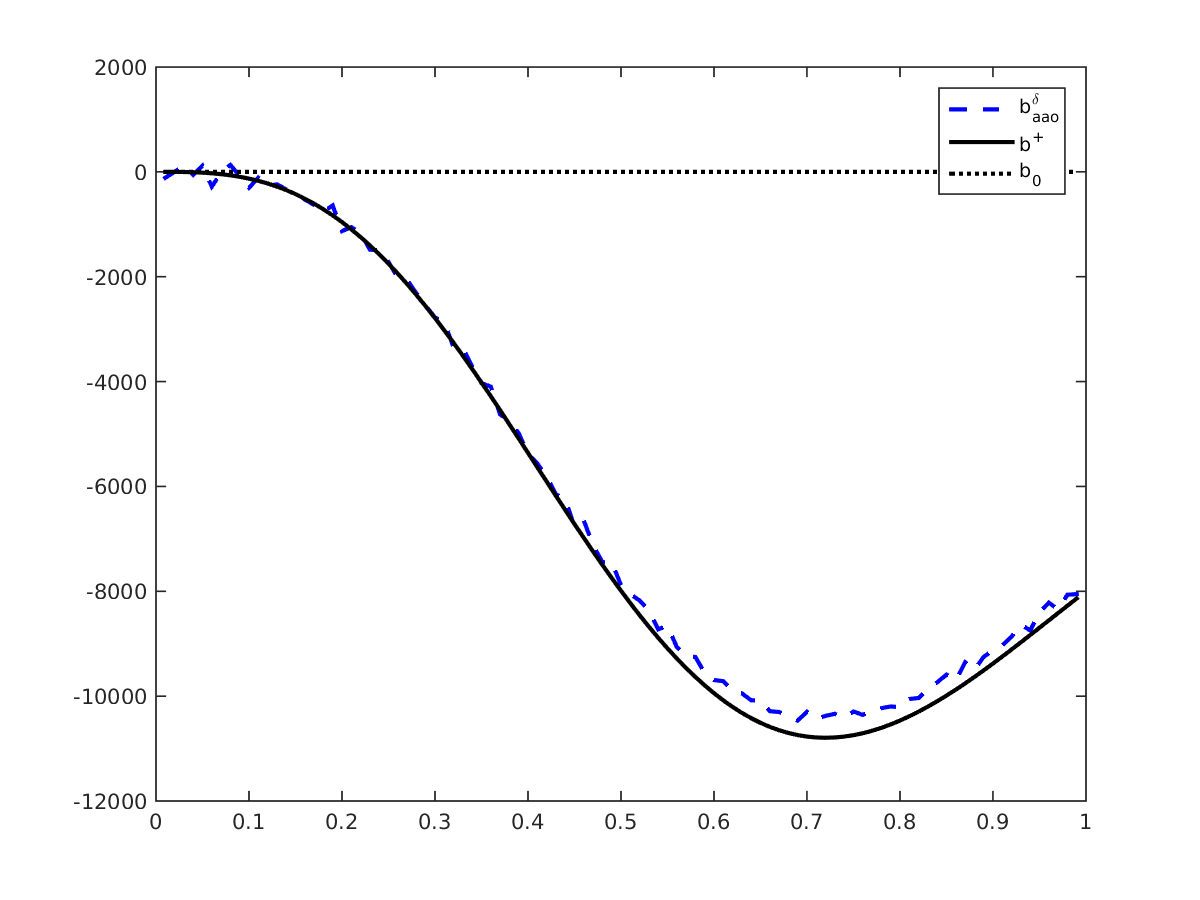}
\\
\includegraphics[width=0.3\textwidth]{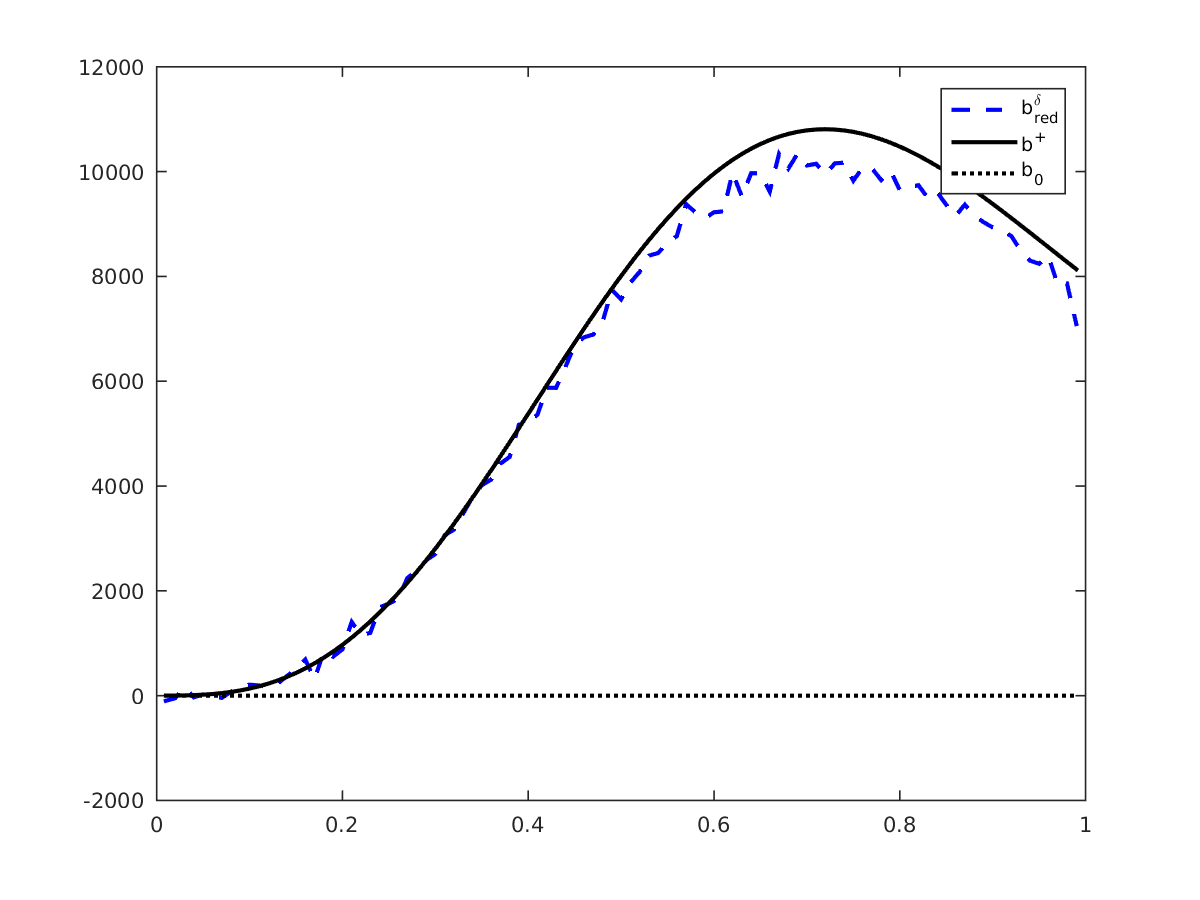}&
\includegraphics[width=0.3\textwidth]{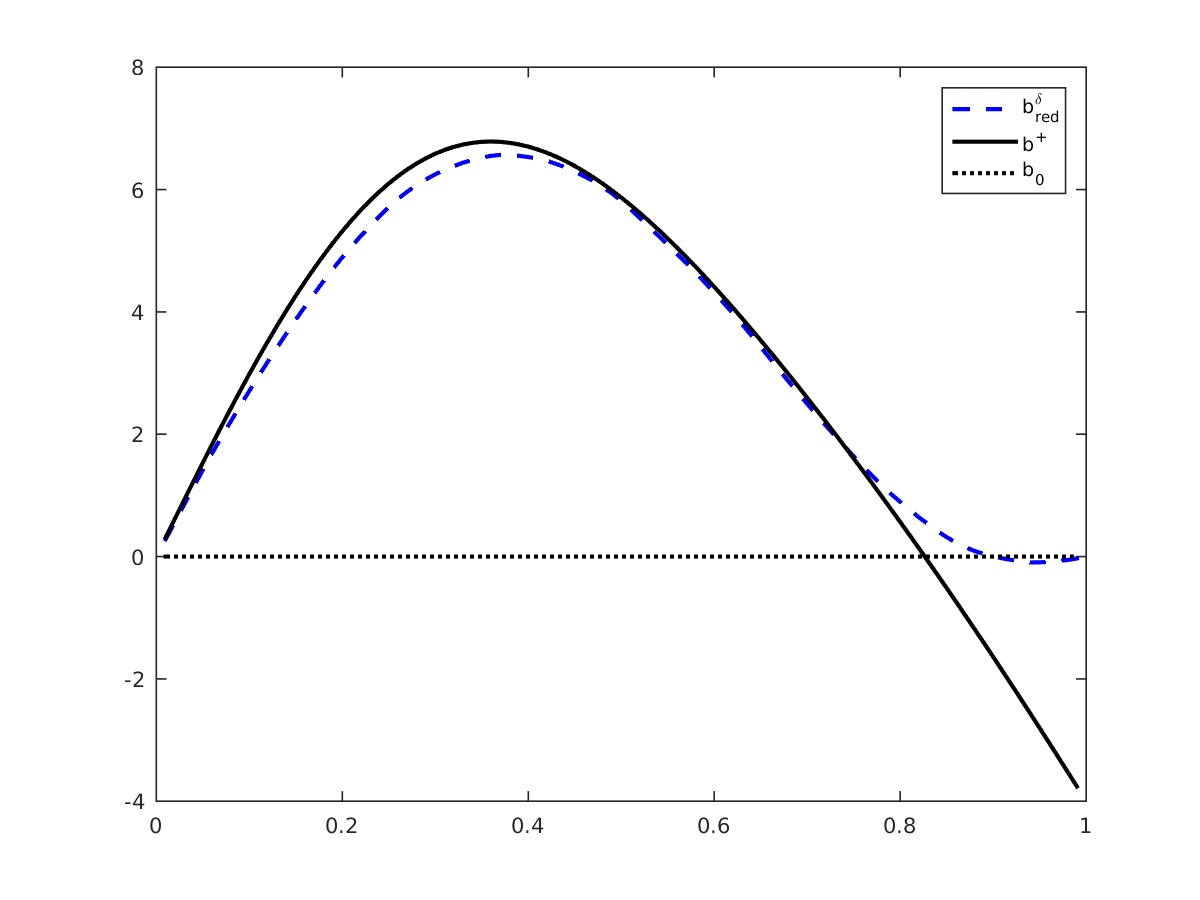}&
\includegraphics[width=0.3\textwidth]{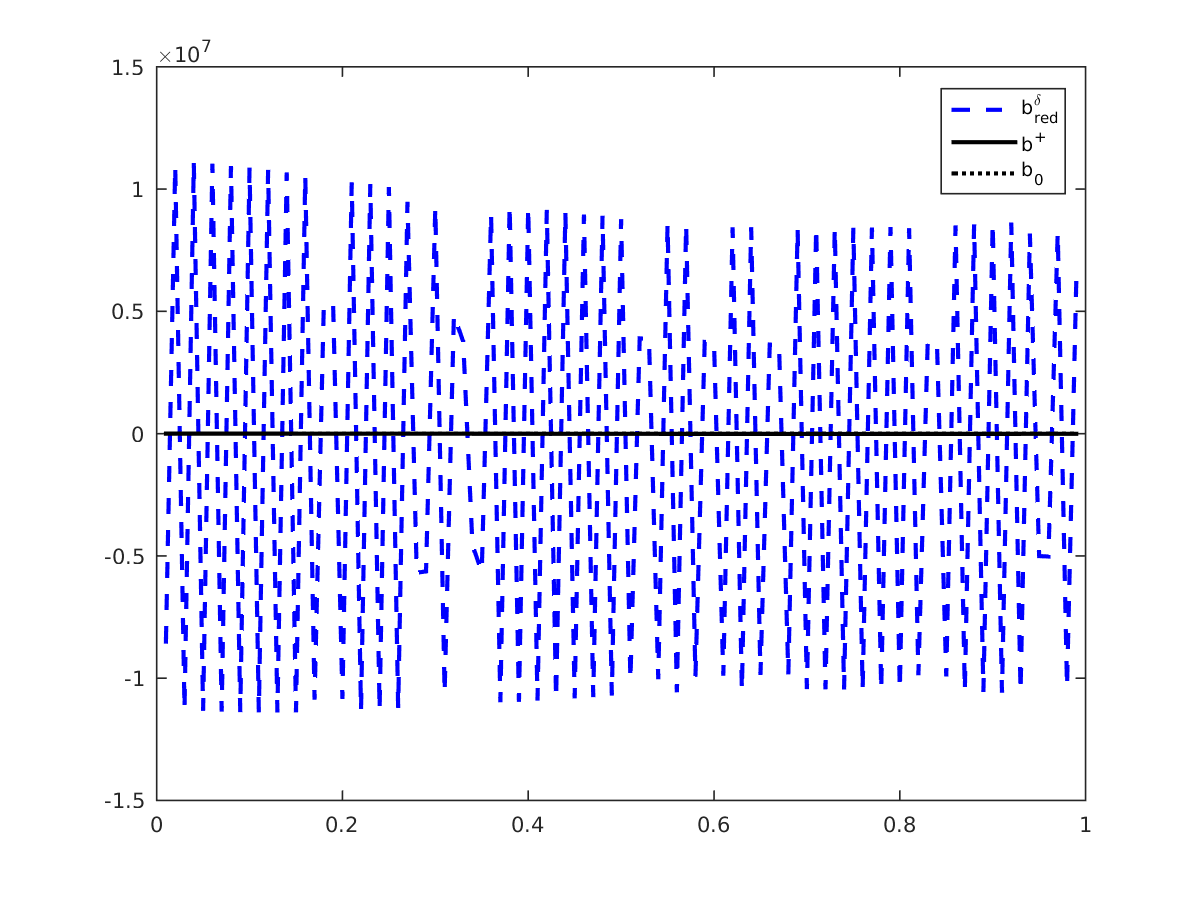}
\end{tabular}
\caption{Comparison of reduced (bottom row) and all-at-once (top row) IRGNM \label{fig:compIRGNM}}
\end{figure}

A similar comparison is done for the  Landweber iteration in Table \ref{tab:compLW} and Figure~\ref{fig:compLW}, however, with less choices of $\xi$ since here already for a relatively moderate (positive) nonlinearity the all-at-once Landweber iteration fails, which is no surprise, since it does not contain solution of models at all. For negative values of $\xi$, like in the IRGNM tests, the reduced version only works with sufficiently small $|\xi|$, whereas the all-at-once version is also able to cope with absolutely slightly larger negative values of $\xi$. 

\begin{table}
\begin{tabular}{|l|l|l|l|l|l|l|l|}
\hline
$\xi$&$\tau^2$&$\mbox{it}_{\mbox{\footnotesize aao}}$&$\mbox{it}_{\mbox{\footnotesize red}}$& $\mbox{cpu}_{\mbox{\footnotesize aao}}$& $\mbox{cpu}_{\mbox{\footnotesize red}}$& 
$\!\!\!\frac{\|b_{k_*(\delta),\mbox{\footnotesize aao}}^\delta-b^\dagger\|_X}{\|b^\dagger\|_X}\!\!\!$ &
$\!\!\!\frac{\|b_{k_*(\delta),\mbox{\footnotesize red}}^\delta-b^\dagger\|_X}{\|b^\dagger\|_X}\!\!\!$ \\
\hline
0.5 & 4 &    5178 &   2697 &    2.97 &  18.07 & 0.0724 & 0.1047\\
 5  & 4 & 2000000 &  48510 & 1293.60 & 482.19 & 0.7837 & 0.1633\\
10  & 4 & 2000000 & 100000 & 1257.50 & 639.87 & 0.9621 & 0.1632\\
-0.5& 4 &   10895 &   2016 &    8.85 &  14.55 & 0.1406 & 0.2295\\
-1  & 4 &   18954 &      - &   11.42 &      - & 0.2313 & - \\  
\hline
\end{tabular}
\caption{Comparison of reduced and all-at-once Landweber \label{tab:compLW}}
\end{table}

\begin{figure}
\begin{tabular}{ccc}
$\xi=5$ & $\xi=0.5$ & $\xi=-0.5$\\
\includegraphics[width=0.3\textwidth]{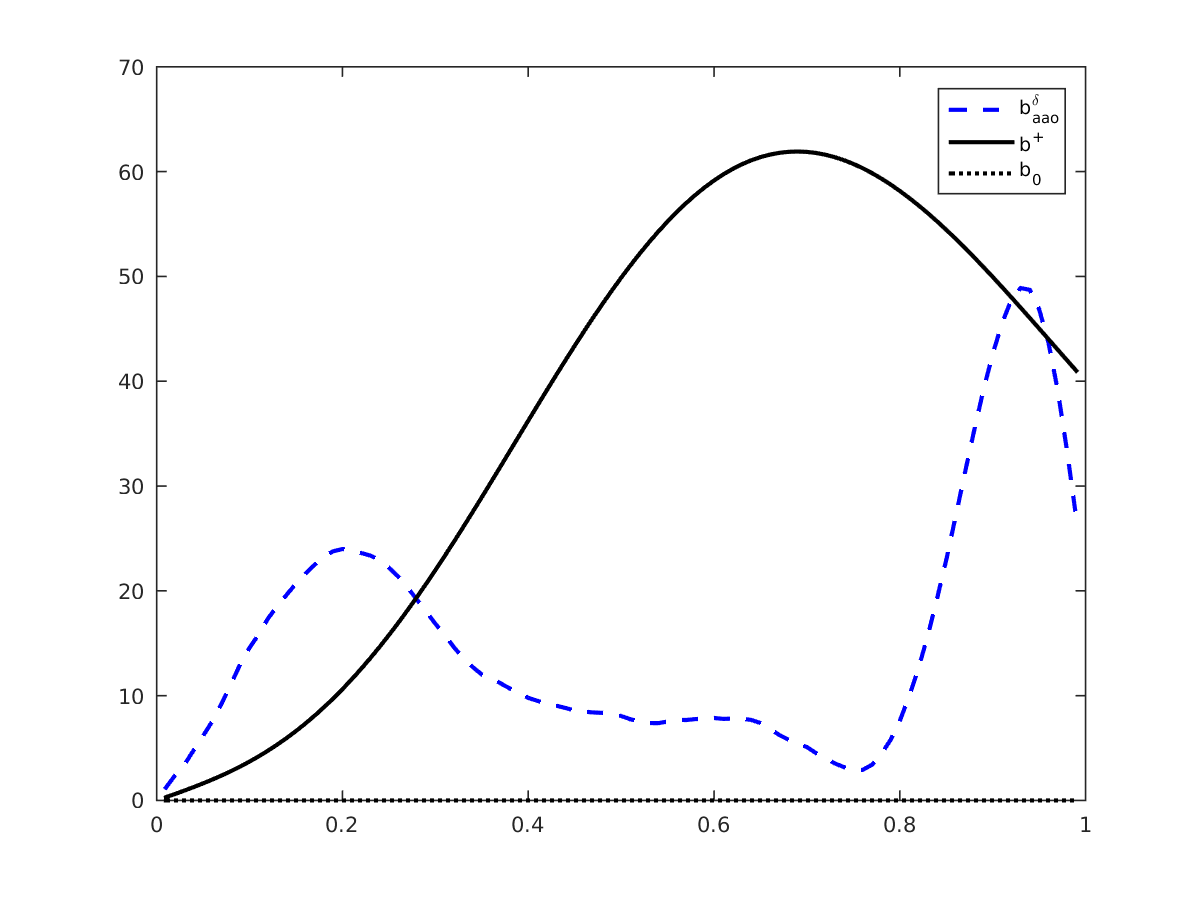}&
\includegraphics[width=0.3\textwidth]{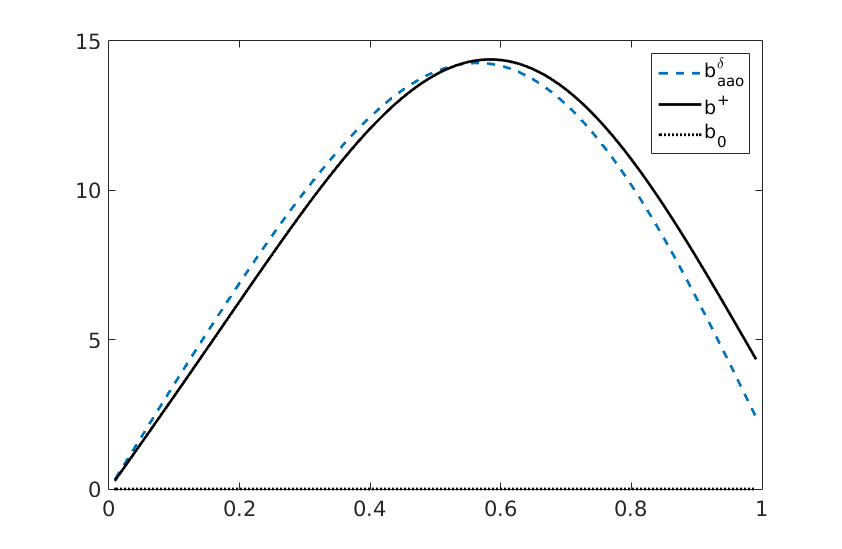}&
\includegraphics[width=0.3\textwidth]{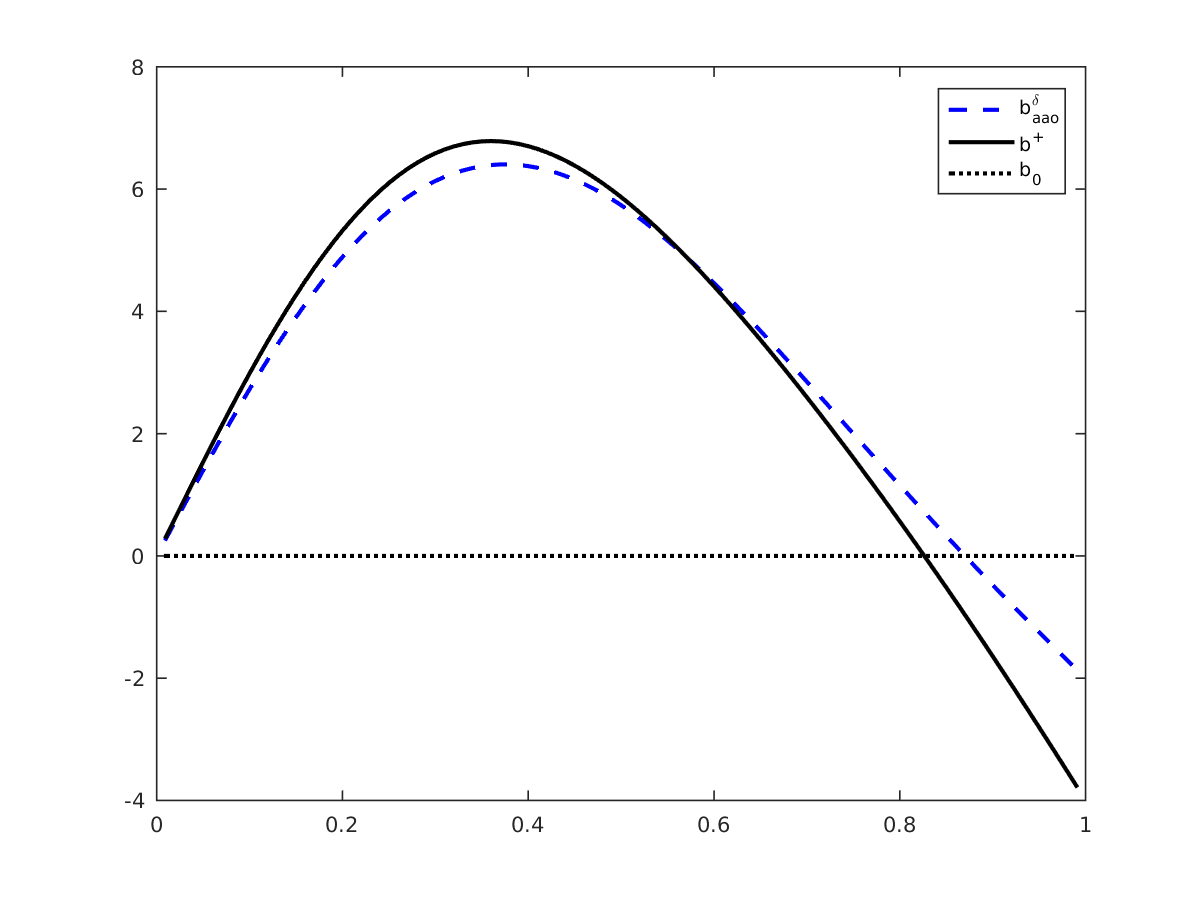}
\\
\includegraphics[width=0.3\textwidth]{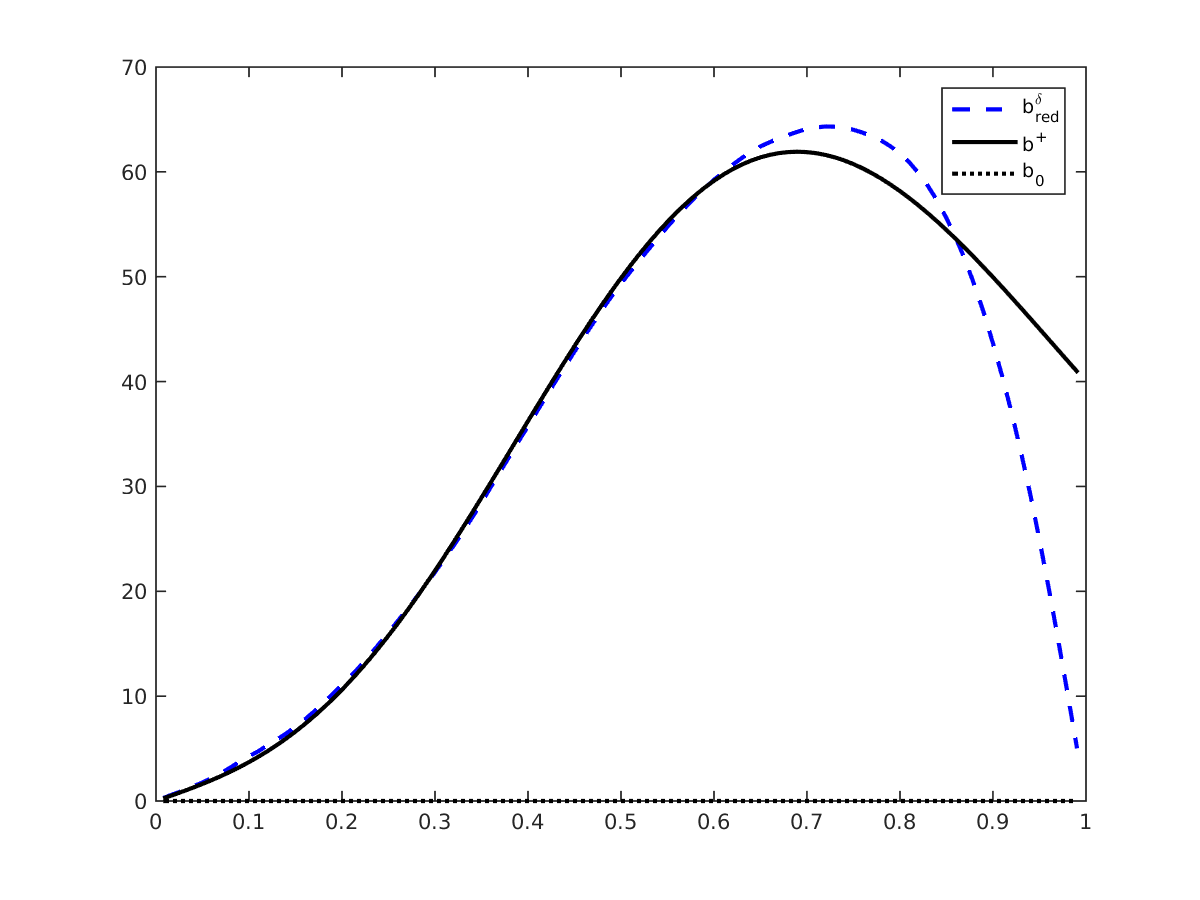}&
\includegraphics[width=0.3\textwidth]{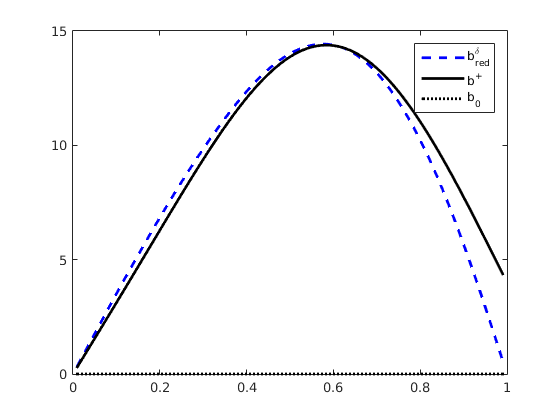}&
\includegraphics[width=0.3\textwidth]{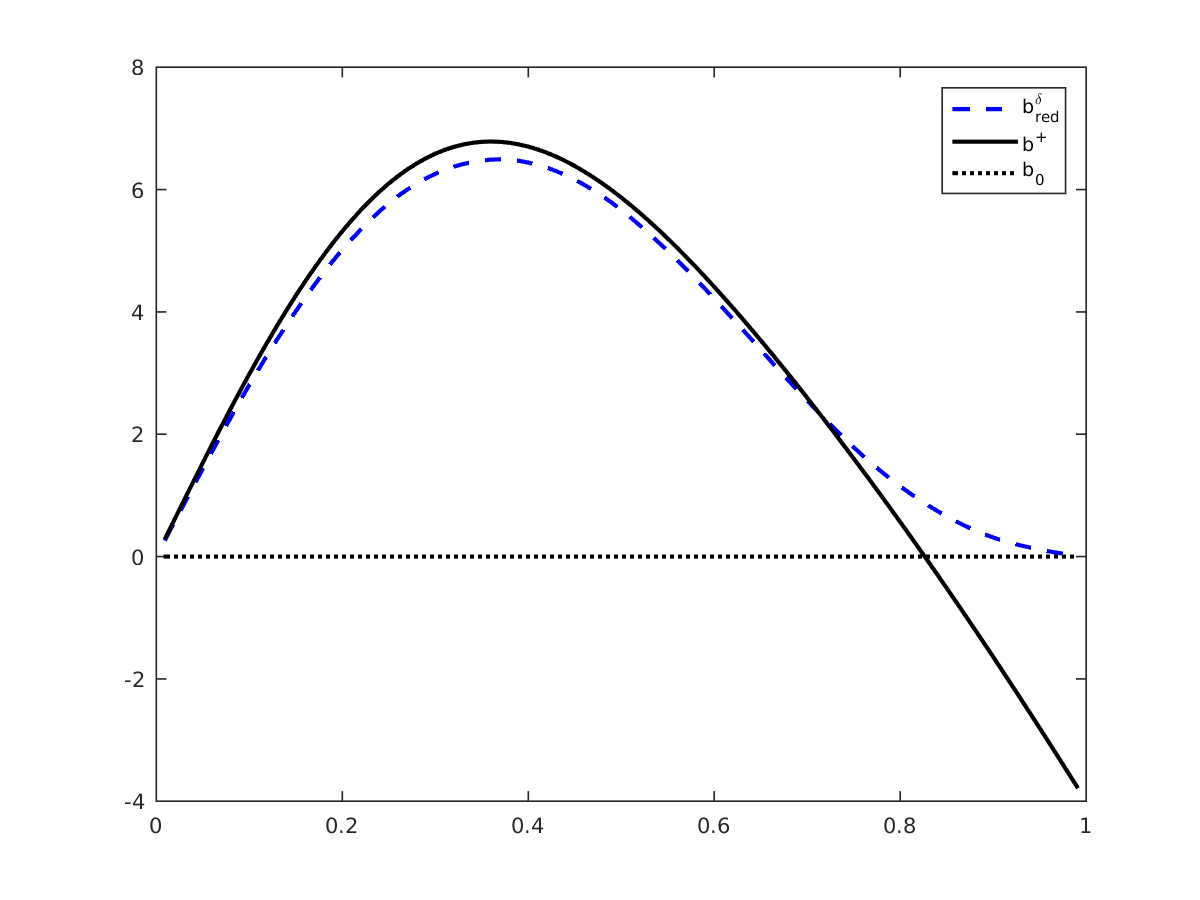}
\end{tabular}
\caption{Comparison of reduced (bottom row) and all-at-once (top row) Landweber \label{fig:compLW}}
\end{figure}

\section{Further remarks on the comparison between all-at-once and reduced version}\label{sec:rem}
\subsection{Comparison of conditions on forward operator}
Assumptions \ref{ass:compactXV}, \ref{ass:bfFweaklyclosed}, and \ref{ass:Lipschitz} cannot be directly compared with their respective reduced versions in the sense that one of them would imply the other. They have to be checked on a case by case basis, which is beyond the scope of this paper but will be subject of future research, e.g., for the examples from the introduction.

As far as the more structural conditions on nonlinearity of the forward operator (tangential cone condition as well as range invariance and adjoint range invariance) are concerned, they seem to be quite different in the reduced and in the all-at-once setting. For instance, both conditions have been proven to hold for certain special cases of Example \ref{ex1} in the reduced setting see, e.g., \cite{EKN89,HaNeSc95,KalNeuSch08} but they do not seem to be rigorously verifiable for these cases in the all-at-once setting. To demonstrate this for the adjoint range invariance condition \eqref{eq:adjrangeinvar} we consider the two special cases 
\begin{itemize}
\item[(i)]
identify $c$, while $a\equiv1$, $b\in H^{-1}(\Omega)$ given, (the so-called $c$ problem), $d\in\{1,2,3\}$;
\item[(ii)]
identify $a$, while $b\in H^{-1}(\Omega)$, $c\equiv0$ given (the so-called $a$ problem), $d=1$;
\end{itemize}
with full observations, i.e., $C$ the embedding operator $V\to Y$.
Note that these are exactly the cases in which the analogous adjoint range invariance condition can be verified for the reduced formulation.

For simplicity of exposition we switch to inhomogeneous Dirichlet instead of Neumann boundary conditions, i.e., 
\[
-\nabla (a\nabla \hat{u})+c\hat{u}=b \mbox{ in }\Omega\,, \quad \hat{u}=g_\Gamma\mbox{ on }\partial\Omega\,,
\]
which with the harmonic extension $g\in H^1(\Omega)$ of the inhomogeneous Dirichlet boundary data $g_\Gamma$ to the interior and $\hat{u}=u+g$, can be formulated variationally as
\[
u\in H_0^1(\Omega) \mbox{ and for all }v\in H_0^1(\Omega)\,:\ \int_\Omega\Bigl( a\nabla(u+g)\nabla v+c(u+g)v\Bigr)=0
\]
where we assume $g_\Gamma$ to be chosen such that 
\[
|u+g|\geq \hat{c}>0 \mbox{ in }\Omega \mbox{ in case (i)  and }
|(u+g)'|\geq \hat{c}>0 \mbox{ in }\Omega \mbox{ in case (ii)}
\]
for all $(x,u)\in\calD(\bfF) \cap \calB_\varrho(x_0,u_0)$, a condition which can be incorporated into the definition of $\calD(\bfF)$.

For the $c$ problem (i) with $x=c$ we make use of higher elliptic regularity and consider the function space setting
\[
X=L^\infty(\Omega)\,, \ V=H^2(\Omega)\cap H_0^1(\Omega)\,, \ W=L^2(\Omega)\,, \ Y= L^p(\Omega)
\]
with $p\in[1,\infty]$.
We have $A_x(x,u)h=h(u+g)$, $A_u(x,u)v=-\Delta v + xv$ and set
\[
\begin{aligned}
&R_{11}w^*=\frac{\tilde{u}+g}{u+g}w^*
&&\quad R_{12}y=\frac{u-\tilde{u}}{u+g}(-\Delta y)+\frac{\tilde{x}(u+g)-x(\tilde{u}+g)}{u+g} y\\
&R_{21}=0&&\quad R_{22}=I 
\end{aligned}
\]
for $w^*\in L^2(\Omega)$, $y\in L^p(\Omega)$.
to formally obtain \eqref{eq:adjrangeinvar}.
It remains to bound the difference $\|\bfR_{(x,u)}^{(\tilde{x},\tilde{u})}-I\|=\sqrt{\|R_{11}-I\|_{W^*\to W^*}^2+\|R_{12}\|_{Y\to W^*}}$, which we do by estimating
\[
\begin{aligned}
\|R_{11}w^*-w^*\|_{W^*}&=\left\|\frac{\tilde{u}-u}{u+g}w^*\right\|_{L^2(\Omega)}
\leq \left\|\frac{\tilde{u}-u}{u+g}\right\|_{L^\infty(\Omega)}\|w^*\|_{L^2(\Omega)}\\
&\leq \frac{C_{H^2\to L^\infty}}{\hat{c}} \|\tilde{u}-u\|_V \|w^*\|_{W^*}
\end{aligned}
\]
for any $w^*\in L^2(\Omega)$ and 
\[
\begin{aligned}
\|R_{12}y\|_{W^*}&=\left\|\frac{u-\tilde{u}}{u+g}(-\Delta y)+\frac{\tilde{x}(u+g)-x(\tilde{u}+g)}{u+g} y\right\|_{L^2(\Omega)}
\end{aligned}
\]
for any $y\in L^p(\Omega)$, which we are supposed to estimate by a small multiple  of $\|y\|_{L^p(\Omega)}$.
However, the appearance of $-\Delta y$ under the $L^2(\Omega)$ norm indicates that although the identity \eqref{eq:adjrangeinvar} formally holds, the required estimate on $\|R-I\|$ in \eqref{eq:ari} does not seem to be obtainable.

For the 1-d $a$ problem (ii) with the somewhat more convenient boundary conditions 
\[
-(a\hat{u}')'=0\mbox{ on }(0,1)\,,\quad \hat{u}'(0)=g^1\,, \quad \hat{u}(1)=g_1\,,  
\]
and under the assumption that $a$ takes the known value $a_0$ at the left hand boundary point, we use the extension $g(s)=g^1 (s-1)+g_1$, $s\in(0,1)$ and the Ansatz $a=a_0+x$, $\hat{u}=u+g$. With 
\begin{align*}
&X=\{x\in W^{1,q}(0,1)\, : \ x(0)=0\}\,,\quad V=\{x\in W^{2,q}(0,1)\cap H^1(0,1)\, : \ v'(0)=0\}\,,\\ 
&W^*=L^q(0,1)\,, \quad Y=L^p(0,1)\,, 
\end{align*}
where the prescribed boundary values are  to be understood in a trace sense and $p\in[1,\infty]$, $q\in(1,\infty)$,
$A(x,u)=-(a_0+x)u''-x'(g^1+u')$, $A_x(x,u)h=-hu''-h'(g^1+u')$, $A_u(x,u)v=-(a_0+x)v''-x'v'$ and setting
\[
\begin{aligned}
&R_{11}w^*=\left(\frac{(\tilde{u}+g)'}{(u+g)'}\int_0^\cdot w^*\,ds\right)'
\hspace*{-0.1cm}&& R_{12}y=-\left(\frac{(a_0+\tilde{x})(u+g)'-(a_0+x)(\tilde{u}+g)'}{(u+g)'}y')\right)'\\
&R_{21}=0
&&R_{22}=I 
\end{aligned}
\]
we formally satisfy \eqref{eq:adjrangeinvar}.
Also here the estimate of the $R_{11}-I$ term of $\bfR-I$ works out
\[
\begin{aligned}
&\|R_{11}w^*-w^*\|_{W^*}=\left\| 
\left(\frac{\tilde{u}''(g^1+u')-u''(g^1+\tilde{u}')}{(g^1+u')^2}\int_0^\cdot w^*\,ds\right)' 
+ \frac{\tilde{u}'-u'}{g^1+u'}w^*
\right\|_{L^q(0,1)}\\
&\leq\Bigl(\frac{1}{\hat{c}}\|\tilde{u}''-u''\|_{L^q(0,1)} +
\frac{1}{\hat{c}^2}\|u''\|_{L^q(0,1)} \|\tilde{u}'-u'\|_{L^\infty(0,1)}\Bigr)\|w^*\|_{L^1(0,1)}\\
&\quad+\frac{1}{\hat{c}}\|\tilde{u}'-u'\|_{L^\infty(0,1)} \|w^*\|_{L^q(0,1)}\,,
\end{aligned}
\]
which by embeddings can be estimated by some constant times $\|\tilde{u}-u\|_V \|w^*\|_{W^*}$,
whereas $R_{12}y$ contains derivatives of $y$ that prevent an estimate of its $L^q(0,1)$ norm by a small multiple of $\|y\|_{L^p(0,1)}$.

\subsection{Comparison of source conditions in reduced and all-at-once setting}
As a consequence of Proposition \eqref{prop:equivTikh} we expect the source condition \eqref{sc_aao_u} to be equivalent for the reduced and the all-at-once version of Tikhonov if $m=1$ (and $\rho$ is large enough). However, for $m>1$, at a first glance there might be a difference, 
so we now consider the special case of $m=o=r=2$ in a Hilbert space setting and compare the reduced and the all-at-once benchmark source condition, i.e., the one yielding $O(\sqrt{\delta})$ convergence of the error norm in both versions.
The reduced source condition 
\[
\xdag-x_0= R_XF'(\xdag)^* v 
\]
for some $v\in Y^*$ with the abbreviations $L=A_x(\xdag,\udag)$, $K=A_u(\xdag,\udag)$ and the Riesz isomorphism $R_X:X^*\to X$ is equivalent to
\begin{equation}\label{sc_red}
\xdag-x_0= -R_X(K^{-1}L)^* C'(\udag)^* v
\end{equation}
whereas the  all-at-once version \eqref{sc_aao_u}
\begin{equation}\label{sc_aao}
\left(\begin{matrix}\xdag-x_0\\\udag-u_0\end{matrix}\right)= 
R_{X\times V}\mathbf{F}'(\xdag,\udag)^* 
\left(\begin{matrix}\vmod\\ \vobs\end{matrix}\right)
= \left(\begin{matrix}R_X L^* \vmod\\
R_V(K^* \vmod+C'(\udag)^* \vobs)\end{matrix}\right)\\
\end{equation}
for some $(\vmod,\vobs)\in W\times Y^*$ after elimination of $\vmod$ (relying on Assumption \ref{ass:A_u}) and using $\udag=S(\xdag)$, $-K^{-1}L=S'(\xdag)$ can be rewritten as
\begin{equation}\label{sc_aao_elim}
\begin{aligned}
&\vmod=(K^*)^{-1}\Bigl(R_V^{-1}(\udag-u_0)-C'(\udag)^*  \vobs\Bigr)\\
&\xdag-x_0+R_X((K^{-1}L)^* R_V^{-1}(K^{-1}L\xdag+u_0)=-R_X(K^{-1}L)^* C'(\udag)^*  \vobs\,.
\end{aligned}
\end{equation}
In the linear case $A_x\equiv L$, $A_u\equiv K$, with $S=-K^{-1}L$, and the positive definite operator $T=I+R_XS^*R_V^{-1}S$ 
the latter is equivalent to 
\begin{equation}\label{sc_aao_lin}
T(\xdag-x_0)= -R_X(K^{-1}L)^* (C'(\udag)^* \vobs+R_V^{-1}(K^{-1}Lx_0+u_0))\,.
\end{equation}
Thus, replacing the linear model $Ku+Lx=0$ by its transformed version $Ku+LT^{-1}\tilde{x}=0$, we see that both source conditions \eqref{sc_red} and \eqref{sc_aao_lin} are equivalent in the Hilbert space case (even with possibly nonlinear observations) provided the initial point satisfies the model. This is actually not surprising in view of well-known converse results for linear inverse problems cf., \cite[Section 4.2]{EHNbuch96} and the fact that, e.g., the respective version of Tikhonov regularization yields the same $O(\sqrt{\delta})$ convergence rate under both conditions.

However, there might still be a considerable difference in the nonlinear and/or Banach space case setting with variational source conditions.
This conjecture is supported by the observation that the nonconvergence at the right endpoint of the interval for the respective reduced version in case $\xi=-0.5$ of Figure \ref{fig:compIRGNM} and cases $\xi=\pm0.5$ of Figure \ref{fig:compLW} (as expected since the difference between $\xdag$ and $x_0$ in this point is known to lead to violation of the reduced benchmark source condition) seems to be relaxed in the all-at-once versions.

\section{Conclusions and Outlook}\label{sec:concl}
All-at- once versions of regularization methods can offer advantages over their classical reduced counterparts when it comes to avoiding explicit use of parameter-to-state maps, i.e., of exactly solving possibly nonlinear models in each step  of iterative methods. More precisely, while there is no significant difference in the implementation of all-at-once and reduced Tikhonov regularization, in reduced Newton iterations one has to  solve nonlinear and linear models in each step, while an all-at-once Newton step only requires to solve linearized models. Still going further, as opposed to reduced Landweber, which amounts to solving a nonlinear and an adjoint linear model in each step, there is no model solved at all in an all-at-once Landweber step.

It remains to more thoroughly compare source conditions and restrictions on the nonlinearity like tangential cone and (adjoint) range invariance  conditions for some relevant model problems and for real applications.
Moreover we will investigate more general data misfit and regularization functionals, as well as other regularization paradigms.

\end{document}